\documentclass[11pt, leqno, a4paper]{article}

\usepackage{times}

\usepackage{verbatim}
\usepackage{hyperref}
\usepackage{amsmath}
\usepackage{enumerate, amsthm}
\usepackage{mathrsfs}
\usepackage{rotating}
\usepackage{xcolor}
\usepackage[margin=1.1 in, top=2.5cm]{geometry}
\usepackage[margin={1.5cm, 1.5cm},font=small, labelsep=endash]{caption}
\usepackage{mathtools}
\usepackage{tikz}
\usetikzlibrary{patterns}
\usetikzlibrary{arrows}
\usepackage{tikz-dimline}
\usepackage{ifthen}

\usepackage{bm}

      \usepackage{amssymb}
      \usepackage{amsmath}
      \usepackage{centernot}

      \numberwithin{equation}{section}
      \theoremstyle{plain}
      \newtheorem{theorem}{Theorem}[section]

            \newtheorem{thm}[theorem]{Theorem}
      \newtheorem{lemma}[theorem]{Lemma}
      
      \newtheorem{corollary}[theorem]{Corollary}
      \newtheorem{proposition}[theorem]{Proposition}
      \newtheorem{prop}[theorem]{Proposition}

      \theoremstyle{definition}

      \theoremstyle{remark}
      \newtheorem{remark}[theorem]{Remark}

\renewcommand{\P}{\mathbb P}
\newcommand{\R}{\mathbb R}
\newcommand{\E}{\mathbb E}

\newcommand{\Z}{\mathbb Z}

\newcommand{\V}{\mathcal{V}}

\newcommand{\lr}[4]{#3\xleftrightarrow[#1]{#2} #4}
     \newcommand{\nlr}[4]{#3\mathrel{\mathop{\centernot\longleftrightarrow}_{#1}^{#2}} #4}

\usepackage{constants}
\usepackage{titletoc}
\usepackage{stackrel}
\usepackage{etex}
\usepackage{enumitem}
\usepackage{enumerate}
\usepackage{vwcol}
\usepackage{stackengine}

\usepackage{hyperref}
\hypersetup{linktocpage,
  colorlinks   = true,
  urlcolor     = blue,
  linkcolor    = blue,
  citecolor   = black
}

\newcommand\couprad{10} 
\newcommand\rangeofdep{40}
\newcommand\pivrange{\the\numexpr\couprad + \rangeofdep\relax}

\newconstantfamily{c}{symbol=c}
\makeatother

\usepackage{titlesec}
\titleformat{\subsection}[runin]{\normalfont\bfseries}{\thesubsection.}{.5em}{}[.]\titlespacing{\subsection}{0pt}{2ex plus .1ex minus .2ex}{.8em}
\titleformat{\subsubsection}[runin]{\normalfont\bfseries}{\thesubsubsection.}{.5em}{}[.]
\titlespacing{\subsubsection}{0pt}{2ex plus .1ex minus .2ex}{.8em}

\usepackage{multicol}
\usepackage{parcolumns}

\usepackage{float}
\usepackage{color}
\definecolor{Red}{rgb}{1,0,0}
\definecolor{Blue}{rgb}{0,0,1}
\definecolor{Yarok}{rgb}{0,0.5,0}

\title{{\textbf{\normalsize{SHARP 
CONNECTIVITY BOUNDS FOR THE %
VACANT SET OF   \\[-0.5em] RANDOM INTERLACEMENTS}}}}
\date{}
\begin{document}
\thispagestyle{empty}
\maketitle
\vspace{0.1cm}
\begin{center}
\vspace{-1.7cm}
Subhajit Goswami$^1$, Pierre-Fran\c{c}ois Rodriguez$^2$ and Yuriy Shulzhenko$^3$

\end{center}
\vspace{0.1cm}
\begin{abstract}
We consider percolation of the vacant set of random interlacements at intensity $u$ in dimensions three and higher, and derive lower bounds on the truncated two-point 
function for all values of $u>0$. These bounds are sharp up to principal exponential order for all $u$ in dimension three and all $u \neq u_\ast$ in higher 
dimensions, where $u_*$ refers to the critical parameter of the model, and they match the upper bounds derived in the article \cite{GRS24.1}. In dimension three, our results further imply that the truncated two-point function grows at large distances $x$ at a rate that 
depends on $x$ only through its {Euclidean} norm, which offers a glimpse of the expected (Euclidean) invariance of the scaling limit at criticality.
The decay rate is atypical, it incurs a logarithmic correction and comes with an explicit pre-factor that converges to $0$ as the parameter $u$ approaches 
the critical point $u_*$ from either side. A particular challenge stems from the combined effects of lack of monotonicity due to the 
{truncation} in the super-critical phase, and the precise ({rotationally invariant}) controls we seek, that measure the effects of a certain ``harmonic humpback'' function. Among others, their derivation relies on rather fine estimates for 
hitting probabilities of the random walk in arbitrary direction $e$, which witness this invariance at the discrete level, and preclude straightforward applications of 
projection arguments. 
\end{abstract}

\vspace{6.0 cm}

\begin{flushleft}
\thispagestyle{empty}
\vspace{0.1 cm}
{\footnotesize
\noindent\rule{6cm}{0.35pt} \hfill {\normalsize June 2026} \\[2em]

\begin{multicols}{2}
\small
 $^1$School of Mathematics\\
 Tata Institute of Fundamental Research\\
 1, Homi Bhabha Road\\
 Colaba, Mumbai 400005, India. \\ \url{goswami@math.tifr.res.in} \columnbreak
 
\hfill $^2$ Center for Mathematical Sciences\\
\hfill University of Cambridge\\
\hfill Wilberforce Road\\
\hfill Cambridge CB3 0WB, UK\\
\hfill\url{pfr26@cam.ac.uk}\\[2em]
\hfill$^3$Imperial College London\\
\hfill Department of Mathematics\\
\hfill 180 Queen's Gate\\
\hfill London SW7 2AZ, UK\\
\hfill \url{yuriy.shulzhenko16@imperial.ac.uk}

\end{multicols}
}
\end{flushleft}

\newpage

\setcounter{page}{1}

\section{Introduction}\label{Sec:intro}
In this article, we investigate the vacant set of random interlacements $(\mathcal V^u)_{u > 0}$ on $\Z^d$ in 
dimensions $d \geq3$ and its percolative properties. As shown in the successive works 
\cite{MR2680403,MR2512613,MR2498684}, the random set 
$\mathcal{V}^u$, which is translation invariant and 
decreasing in $u$ (by means of a suitable coupling of the family $(\mathcal V^u)_{u > 0}$, see \cite{MR2680403}), undergoes a percolation phase transition across a non-degenerate 
threshold $u_*=u_*(d) \in (0,\infty)$ for all $d \geq 3$. This transition entails that whenever 
$u<u_*$, there exists a unique infinite cluster in $\mathcal{V}^u$ with probability one. In  
contrast, for all $u> u_*$ the connected components (clusters) of $\mathcal{V}^u$ are finite 
almost surely.  

Our focus in this article is on the so-called {truncated two-point 
 function} of $(\mathcal V^u)_{u > 0}$, 
\begin{equation}\label{def:truncated_twopt}
\tau_{u}^{{\rm tr}}(x, y) \stackrel{\text{def.}}{=} \P[\lr{}{\mathcal V^u}{x}{y}, \nlr{}{\mathcal V^u}{x}{\infty}],\, \mbox{for $x, y 
\in \Z^d$ and $u \in (0, \infty)$},
\end{equation} 
where, with hopefully obvious notation $\lr{}{C}{A}{B}$ refers to the existence of a nearest-neighbor path in $C$ connecting $A$ and $B$, and the symbol $\nlr{}{}{}{}$ to its complement. The function $\tau_{u}^{{\rm tr}}$ is symmetric, i.e.~$\tau_{u}^{{\rm tr}}(x, y) =\tau_{u}^{{\rm tr}}(y, x) $, and $\tau_{u}^{{\rm tr}}(x, 
 y)= \tau_{u}^{{\rm tr}}(0, y-x) \equiv \tau_{u}^{{\rm tr}}(y-x)$ for all $x,y \in \Z^d$ by translation invariance of 
 $\mathcal{V}^u$. When $u>u_*$ the `truncation' event that $x$ belongs to a finite component of $\mathcal{V}^u$ has probability one and can be safely omitted, so that 
$\tau_{u}^{{\rm tr}}(x)=\tau_u(x)$, the usual two-point 
function, i.e.~$\tau_u(x)$ refers to the probability to connect $0$ and $x$ in  $\mathcal V^u$.  When $u<u_*$ however, by the FKG-inequality, which holds~\cite{MR2525105}, one has that $\tau_u(x) \geq \P[\lr{}{\mathcal V^u}{0}{\infty}]^2 >0$, 
which does not decay at all, and the presence of the disconnection event in \eqref{def:truncated_twopt} is crucial for $\tau_{u}^{{\rm tr}}$ to be small at values of $u<u_*$.

Our aim in this article is to determine sharp asymptotics for $\tau_{u}^{{\rm tr}}(x)$ 
at large distances $|x|$ for any value $u \neq u_*$; the hastened reader is referred to Corollary~\ref{cor:sharpB} below. This question has already instigated a number of previous works, which we now briefly summarize. Deep in the sub-critical regime, when $u \gg 1$, polynomial upper bounds on $\tau_u(\cdot)$ follow from the results of~\cite{MR2680403}, and were first improved to stretched exponential bounds in \cite{MR2744881} for all values $u \geq u_{**}$ with $u_{**} \geq u_*$ an auxiliary parameter introduced in \cite{sznitman2009}. Later, in \cite{PopTeix}, $\tau_u(\cdot)$ was proved  to decay exponentially for all $u \geq u_{**}$ in dimensions four and higher, with (sub-optimal) logarithmic corrections in dimension three. The logarithmic correction for $u>u_{**}$ was improved much more recently in \cite{prevost2023passage}, and this improvement is sharp on account of the results of the present article. The findings of \cite{prevost2023passage} also rely on certain fine estimates developed in \cite{MR3602841} in a different context, and are inspired by corresponding results for the Gaussian free field \cite{gosrodsev2021radius,MR3417515}. As a consequence of the sharpness result  \cite{RI-I} (see in particular (1.21) therein), drawing upon the companion 
articles \cite{RI-II, RI-III}, the equality $u_*=u_{**}$ was established, by which all previous results now hold throughout the sub-critical phase $u>u_*$.

Owing to the presence of the truncation in \eqref{def:truncated_twopt}, the super-critical phase $u< u_*$ is even more intricate -- indeed it will also be the main focus of this paper. For $u \ll 1$ and $d \geq 5$, stretched exponential bounds for $\tau_{u}^{{\rm tr}}$ follow from the results of \cite{Tei11}. This was extended to all dimensions $d \geq 3$ and $u \ll 1$ in \cite{MR3269990}. Until recently, the closest to pushing these results to all values of $u< u_*$ are the results of \cite[Theorem~1.2]{RI-I}, which together with the disconnection estimate in \cite{MR2498684} 
yield a stretched exponential bound on a related but {\em strictly smaller} quantity than $\tau_u^{{\rm tr}}$, where the 
disconnection in \eqref{def:truncated_twopt} happens at an intensity $v < u$. The sprinkling is a significant drawback. In very recent work \cite{GRS24.1} we managed to remove this sprinkling, i.e.~set $u=v$ to obtain upper bounds on $\tau_{u}^{{\rm tr}}$ for all values $u<u_*$. The main 
object of this paper is to derive matching lower bounds for $\tau_{u}^{{\rm tr}}$, thus ascertaining the sharpness of results of \cite{GRS24.1}. In the sequel, 
we write $|\cdot|$ for the Euclidean distance on $\Z^d$.
 \begin{thm}\label{T:LB-RI}
 For all $u \in (0, \infty)$,
 \begin{equation}\label{eq:LB_RI-1}
\liminf_{|x | \to \infty}\frac{\log |x |}{|x|}\log \tau_u^{{\rm tr}}(x) \geq  -\frac{\pi}{3}(\sqrt{u} - 
\sqrt{u_*})^2, \quad d=3.
\end{equation}
When $d  \ge 4$, for all $u \in (0,\infty)$, there exists $C = C(u, d) \in (0, \infty)$ such that for all $x \in \Z^d$,
\begin{equation}\label{eq:LB_RI-2}
 \log \tau_u^{{\rm tr}}(x) \geq -C|x|.
\end{equation}
 \end{thm}
Combined with Theorem~1.4 in the companion article \cite{GRS24.1}, this implies
\begin{corollary}\label{cor:sharpB}
For all $u \in (0, \infty)$, 
 \begin{equation}\label{eq:SB_RI-1}
\lim_{|x | \to \infty}\frac{\log |x |}{|x|}\log \tau_u^{{\rm tr}}(x) =  -\frac{\pi}{3}(\sqrt{u} - 
\sqrt{u_*})^2, \quad d=3.
\end{equation}
When $d  \ge 4$, for all $u \neq u_*$, there exist $C = C(u, d)$, $c = c(u, d)$ in $(0, \infty)$ such that for all $x \in \Z^d$,
\begin{equation}\label{eq:SB_RI-2}
-c \, |x| \ge  \log \tau_u^{{\rm tr}}(x) \geq - C \, |x|.
\end{equation}
\end{corollary} 

 We now make a few comments about Theorem~\ref{T:LB-RI}, focussing entirely on \eqref{eq:LB_RI-1} in the super-critical regime $u<u_*$, which is the main interest of this article. The other results are obtained via simplified versions of these arguments. The methods  developed in this article can in fact be adapted to
prove an analogue of Corollary~\ref{cor:sharpB} for excursion sets of the discrete Gaussian free field (GFF) on $\Z^3$, as conjectured in \cite[(1.11)]{gosrodsev2021radius}. 

There are two main obstacles to proving \eqref{eq:LB_RI-1}, and, a fortiori, \eqref{eq:SB_RI-1}. One issue is the non-monotonic nature of $u  \mapsto \tau_u^{{\rm tr}}(x)$. The other issue is the Euclidean invariance of the lower bound on the right-hand side of \eqref{eq:LB_RI-1}. A result akin to Corollary~\ref{cor:sharpB} is known to hold for the metric-graph GFF in three dimensions, see \cite[Corollary 1.3]{drewitz2023arm} and combine with \cite{arXiv:2405.17417} or \cite{arXiv:2406.02397} to remove the $\log \log$ correction. However, the combination of the above two issues can be avoided for this model because bounded metric-graph GFF clusters on either side of the critical point have the {\em same} distribution, see  \cite{zbMATH07529630}, so it is effectively enough to deal with the sub-critical problem, corresponding to $u>u_*$; see also \cite{MR4749810,prevost2023passage} for further results of this flavour, all in sub-critical regimes.
 
As we explain in more detail below, the non-monotonicity will be dealt with by extending a tilting technique pioneered by Li and Sznitman \cite{zbMATH06797082,zbMATH06257634}, who considered the (monotone) disconnection problem. The extension will consist of also keeping track of the time parameter $u$ associated to trajectories, rather than just the trajectories themselves, which is carefully modulated but in opposite directions in distinct spatial regions (cf.~the discussion around \eqref{eq:LU-intro2}), so as to ``identify'' (via coupling arguments) regions with different effective values of $u$ after tilting (again, cf.~\eqref{eq:LU-intro2} below). The carefully designed profile function $f$ used to spatially modulate the level $u$ is depicted in Fig.~\ref{fig:profile}. It looks like a ``harmonic humpback'' when $u<u_*$; the simplifications incurred when $u>u_*$ are mirrored by the fact that $f$ becomes monotone away from the middle plateau in this regime.

For the purposes of producing rotationally invariant bounds, the definitions of the relevant spatial regions will involve discretisations of {\em oblique} `corridors' (around $x$ and $y$), symmetric around $e=\tfrac{x-y}{|x-y|}$, in which the connection between $x$ and $y$ appearing in \eqref{def:truncated_twopt} can be made sufficiently cost-efficient. The implementation of this strategy will require rather delicate (pointwise) estimates on hitting and escape probabilities for such sets. We deal with these issues separately in Section~\ref{sec:oblique}. Our approach yields robust arguments that involve certain (non-standard) martingales, cf.~Fig.~\ref{fig:mart} and the proof of 
 Proposition~\ref{L:hit}. These are of independent interest.
 
 \medskip

 \begin{figure}[htb!]
 		 \centering
 		\includegraphics[width=\linewidth]{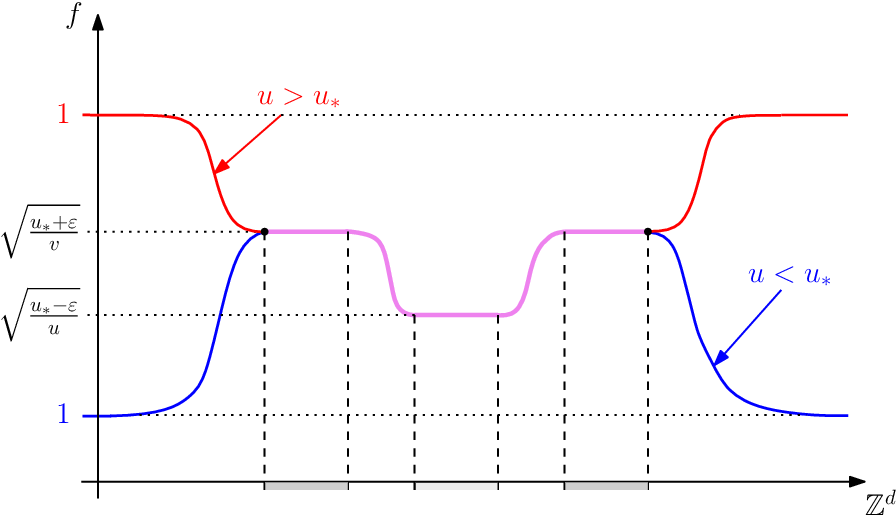}
 		\caption{{\bf The ``harmonic humpback'' function $f$.} The purple part is joint to both the red and blue curve, which distinguish sub-critical (red) and super-critical (blue) regimes. The argument of $f$ is a {ratio} of intensities, and the value $f=1$ in red or blue corresponds to the critical point $u_*$ in each regime. In reality, each curve depicts a section orthogonal to the direction $e \in S^{d-1}$, which is arbitrary. The plateaux correspond to spatial regions (in shaded grey) at which trajectories arrive at roughly constant (but suitably modulated) intensity. For practical purposes the reader should think of $v \approx u$.}
 		\label{fig:profile}
 	\end{figure}

 We now highlight some aspects of the proof of Theorem~\ref{T:LB-RI}, and provide a guide to how subsequent sections are organized.
 Theorem~\ref{T:LB-RI} will follow from a more general result, Theorem~\ref{T:LB-gen}, which yields a lower bound on the probability of a class of events $A^u(x,\delta)$, that, roughly 
 speaking, allow to connect the $|x|^{\delta}$-neighborhoods of $0$ and $x \in \mathbb{Z}^d$ by a sufficiently 
 ``straight'' cluster (i.e.~efficiently in diameter) in the vacant set $\mathcal{V}^u$, which is 
 isolated from infinity (eventually, $\delta \downarrow 0$). The event $A^u(x,\delta)$ is in fact 
 more involved, it also entails the flexibility to perform local surgeries near $0$ and $x$ 
 (borrowing a technique recently introduced in \cite{RI-II}), which normally would follow by 
 application of the FKG-inequality, which however fails to apply when the event in question is 
 not monotone.

 Theorem~\ref{T:LB-gen} has a lot of mileage, in allowing us to deduce 
 Theorem
 ~\ref{T:LB-RI}, along with other interesting results, among 
 which, bounds on the so-called truncated one-arm events, see Remark~\ref{R:1-arm}.
 Another case in point is the following `local uniqueness' event $\textnormal{LU}$ and variants thereof, which play a prominent role in the proof: for $u,v,\varepsilon > 0$, with $B_r$ denoting the  ball of radius $r>0$ around $0$, let
  \begin{equation}\label{eq:LU-intro}
  \text{LU}_{r,\varepsilon}^{u,v} \text{ ``$=$'' } \left\{\!\begin{array}{r} (B_{(1+\varepsilon)r} \setminus B_{r/2}) \cap \mathcal{V}^{u}  \text{ has at least one crossing component} , \\  (B_{r} \setminus B_{r/2})\cap \mathcal{V}^{v} \text{ has at most one crossing component} \end{array} \right\},
  \end{equation}
  where ``crossing component'' refers to a connected set intersecting both boundary components of the annulus in question. The previous display should be read with a grain of salt (see around \eqref{eq:typical-22} for the precise definition). The problem with \eqref{eq:LU-intro} as such is that the event $\text{LU}_{r,\varepsilon}^{u,v}$ is not monotone in the configuration $\mathcal{V}^{v}$. For the purpose of this introduction, it will be sufficient to have \eqref{eq:LU-intro} in mind when referring to $\text{LU}_{r,\varepsilon}^{u,v}$ in the sequel.
Events of this type 
are 
instrumental in 
understanding the properties of $\mathcal V^u$ in the super-critical phase; see, 
e.g.~\cite{MR3602841, https://doi.org/10.48550/arxiv.2105.12110, zbMATH07226362, RI-I, 
RI-II, li2024sharp}.

The study of  $\text{LU}_{r,\varepsilon}^{u,v}$ for $v<u$ (and in fact the more general event considered in Theorem~\ref{T:LB-gen})
proceeds through a change of probability method, that involves a carefully designed tilt 
$\widetilde{\P}_f$ of the canonical law $\mathbb{P}$ of the interlacement point process. As 
within classical large deviation theory, the gist is for the change of measure to simultaneously 
make the event of interest likely (i.e.~with probability of order unity rather than, say, the 
right-hand side of \eqref{eq:LB_RI-1}), all the while retaining good control on the 
Radon-Nikodym derivative it induces, which eventually leads to lower bounds as in 
\eqref{eq:LB_RI-1}. A similar but simpler tilting technique than the one we employ here was 
developed in a series of articles \cite{zbMATH06257634, zbMATH06797082,MR4650161,arXiv:2312.17074, arXiv:2405.14329} to study various monotone events of interest. In the present case, the function $f:\Z^d \to \mathbb{R}$ induces a tilt on \textit{labelled} 
trajectories entering the interlacement set which acts as a Doob-transform for the 
trajectories, in a manner that depends on the label (here working with labeled trajectories is 
intimately linked to the fact that the event \eqref{eq:LU-intro} requires working simultaneously 
at different levels), in order to induce a propitious spatially modulated level $u(x)$, for which 
we now give some heuristics.

Roughly speaking, 
the effect of the function $f$ we choose is  to create a `corridor' $T$ and concentric interface $T'$ (see the regions in light and dark grey on the horizontal axis in Fig.~\ref{fig:profile}), in which
 the new vacant sets $\widetilde{\V}^u$, resp.~$\widetilde{\V}^v$ declared under $\widetilde{\P}_f$ look slightly super-, resp.~sub-critical. That is, $f$ is designed so as to roughly ensure that for any $0<v \leq u$, and small $\varepsilon > 0$
 \begin{equation}\label{eq:LU-intro2}
 \widetilde{\V}^u \vert_{T} \stackrel{\text{law}}{\approx} \V^{u_* - \varepsilon}, \quad  \widetilde{\V}^v \vert_{T'} \stackrel{\text{law}}{\approx} \V^{u_* + \varepsilon}.
 \end{equation}
 The rationale behind \eqref{eq:LU-intro2} is that the opposite requirements appearing in \eqref{eq:LU-intro} become typical at the effective new levels $u_*\pm \varepsilon$; indeed connection in $\mathcal{V}^u$ is likely below $u_*$, whereas disconnection is above $u_*$. The way \eqref{eq:LU-intro2} is made precise is by developing couplings allowing to compare tilted and untilted interlacements; see Proposition~\ref{P:coupling}. These couplings boil down to rather precise comparison estimates between tilted and untilted harmonic quantities for the random walk (see for instance Proposition~\ref{P:harm-comp}) below, which make the ``effective levels'' $u_* \pm \varepsilon$ appear, and are ultimately responsible for our choice of tilting function $f$. One technical point is that the tilting function $f$ is typically not the discrete blow-up of a continuous analogue, which makes proving these comparison estimates somewhat involved.

 The remaining part is to control the relative entropy $H(\widetilde{\mathbb{P}}_{f} \vert \P)$ 
 between $\widetilde{\mathbb{P}}_{f} $ and $\P$, and with it the derivative between the two 
 measures. This eventually follows from (killed) capacity estimates for corridor regions $T$ 
 and $T'$, somewhat similar in spirit to those that arose in \cite{gosrodsev2021radius} in the 
 context of the Gaussian free field. However, unlike the setup of \cite{gosrodsev2021radius}, 
 which dealt with one-arm events (in the present context this would mean choosing $T,T'$ to be axis-aligned), the efficient connection between $0$ and $x$ underlying 
 \eqref{eq:LB_RI-1} follows the \textit{Euclidean} distance, hence the sets $T$ and $T'$ 
 are oblique and directed towards $e=\frac x{|x|}$. This setup is paramount in order to yield 
 \eqref{eq:LB_RI-1} and it precludes the use of certain projection arguments (onto subsets of 
 coordinates) used in \cite{gosrodsev2021radius}. Rather than following a route suggested in 
 \cite[Remark 5.17,2)]{gosrodsev2021radius}, which would resort to dyadic coupling to 
 Brownian motion, we develop robust techniques, which could be adapted, e.g.~to the setup 
 of  \cite{drewitz2018geometry} and involve the afore mentioned (non-standard) martingale arguments.
 In essence, the martingales in questions are built summing the Green's function over certain carefully designed sets, so as to ensure a desirable uniformity in the ensuing estimates; see Figure~\ref{fig:mart} for an illustration.
 
 \medskip
 
We now describe how this article is organized. 
Section~\ref{sec:oblique} 
isolates 
estimates on the hitting probability for an obliquely aligned cylinder pertaining to the problem 
discussed above. 
In Section~\ref{sec:gen} we state our general lower bound in Theorem~\ref{T:LB-gen} and 
deduce from it our main result, 
namely Theorem~\ref{T:LB-RI} above. The rest of the paper is devoted to the proof of 
Theorem~\ref{T:LB-gen}. Section~\ref{sec:tilt} introduces the tilting functions $f$ of interest, 
along with the tilted interlacement measure $\widetilde{\P}_f$ they induce. 
Theorem~\ref{T:LB-gen} is then reduced to two results, Propositions~\ref{P:entropy} 
and~\ref{P:typical}, from which the proof of Theorem~\ref{T:LB-gen} is readily concluded. 
Proposition~\ref{P:entropy} is the control on the relative entropy 
$H(\widetilde{\mathbb{P}}_{f} \vert \P)$. It is proved in Section~\ref{sec:entropy} which draws 
upon, among others, the aforementioned hitting probability estimate for `oblique corridors' 
given by Proposition~\ref{L:hit} in Section~\ref{sec:oblique}. Proposition~\ref{P:typical} 
asserts that the event of interest in Theorem~\ref{T:LB-gen} becomes typical under the tilted 
measure. The proof of Proposition~\ref{P:typical} spans Sections~\ref{sec:coupling} 
and~\ref{sec:tilt-hm}. The key ingredient is a coupling between tilted and untilted 
interlacements at suitable levels that makes precise the above heuristics, in particular, \eqref{eq:LU-intro2}.

 
 \bigskip
 
 Our convention regarding constants is as follows. Throughout the article $c, c', C, C', \ldots$ denote generic 
 constants with values in $(0, \infty)$ which are allowed to change from place to place. All constants may implicitly depend on the 
 dimension $d \ge 3$. Their dependence on other parameters will be made explicit. Numbered constants are 
fixed when 
first appearing  
within the text.

\section{Hitting probabilities for oblique sets}\label{sec:oblique}

We first 
present an 
upper bound on the hitting probability of a cylinder (referred to as {\em corridor} in the introduction) 
aligned along an {arbitrary} direction in 
$\R^3$. This bound 
becomes effective when the 
distance between the 
starting point of the walk 
and the cylinder is of {lower order} than its {length}, 
which renders 
the `standard' 
estimate obtained using bounds on capacity and Green's function (see, e.g.~via \eqref{eq:last-exit} below) 
{\em trivial} (i.e. $\ge 1$). A similar estimate was obtained in 
\cite[(2.25)]{gosrodsev2021radius} for {\em axis-aligned} cylinders. The main feature of our 
bound is that it holds {uniformly} in all directions which is consistent with the rotational 
invariance apparent in \eqref{eq:LB_RI-1} and \eqref{eq:SB_RI-1}. 
The argument in \cite{gosrodsev2021radius} relied heavily on projection arguments, whose 
straightforward application is precluded when the 
cylinders are oblique. 
We proceed using a (different) martingale argument which inherits the required (rotational) 
symmetry from the 
large-distance behavior of the Green's function (see~\eqref{eq:c_2} below) rather than from
the Brownian motion directly, as suggested in \cite[Remark 
5.17,2)]{gosrodsev2021radius}. 
Consequently, our method could be adapted to very general class of graphs; see, 
e.g.~\cite{drewitz2018geometry}.


Below and in the remainder of the article, 
we write $P_x$ for the canonical law of the continuous-time random walk on $\Z^d$, $d \geq 
3$, with starting point $x\in \Z^d$ and mean one exponential holding times. We write 
$X=(X_t)_{t \geq0}$ for the canonical process under $P_x$. For $K \subset \Z^d$, we 
introduce the stopping time $H_K = \inf\{t\geq 0: X_t \in K\}$, $T_K=H_{\Z^d \setminus K}$ 
and $\widetilde{H}_K=\inf\{ t \geq 0: X_t \in K \text{ and } \exists s\in [0,t] \text{ s.t. } X_s\neq 
X_0\}$. For $U \subset \Z^d$, we write $g_U$ for the Green's function of the walk killed when 
exiting $U$, i.e.
\begin{equation}
	\label{eq:g_U}
	g_U(x,y)= E_x\Big[ \int_0^{T_U} 1\{X_t=y \} \mathrm{d}t \Big], \ x,y\in \Z^d
\end{equation}
which is symmetric and finite, and for any $K\subset U$ (finite) we denote by $e_{K,U}$ for 
the equilibrium measure of $K$ relative to $U$, \begin{equation}
	\label{eq:e_K,U}
	e_{K,U}(y)=P_y[\widetilde{H}_K > T_U] 1_{\{y \in K\}}.
\end{equation}
Its total mass is denoted by $\text{cap}_U(K)$, the capacity of $K$ (relative to $U$). We 
omit $U$ from the notations in \eqref{eq:g_U} and~\eqref{eq:e_K,U} whenever $U=\Z^d$ 
(with $T_{\Z^d} = \infty$ by convention). An application of the strong Markov property yields 
the formula
\begin{equation}
	\label{eq:last-exit}
	P_y[H_K< T_U]= \sum_{z \in K} g_U(y,z)e_{K,U}(z), 
\end{equation}
valid for all $y \in \Z^d$ and $K$ finite set.

We write $|\cdot|$ for the Euclidean norm and denote by $d(\cdot,\cdot)$ the Euclidean distance between sets.
Let $x \in \Z^d \setminus \{0\}$ 
and $e=\frac{x}{|x|} $. For a point $z\in 
\mathbb{R}^d$, let $[z] \in \mathbb{Z}^d$ be a point achieving $d(z, \mathbb{Z}^d)$. 
We now introduce certain discretized cylindrical sets, namely (with  $e=\frac{x}{|x|} $)
\begin{equation}\label{eq:T_x}
	\begin{split}
		&T(x)= T(x,0)= \{ [j
		e] : 0 \leq j \leq \lceil |x|
		\rceil \},\\
		&T(x,r) = \{ y \in \Z^d: \, d(y,T(x))\leq r\}, \quad r \geq0.
	\end{split}
\end{equation}
In the sequel, $\partial K \stackrel{}{=} \{x \in K: d(x, K^c) = 1\}$ denotes the inner 
vertex boundary of $K \subset \Z^d$ 
whereas $\partial^{\text{out}}K = \partial(\Z^d \setminus K)$ its outer 
boundary.
 \begin{prop}
 \label{L:hit} For all 
$\delta, \varepsilon \in (0, 1)$, 
$|x| \geq C(\delta, \varepsilon)$, if $d=3$,
 	\begin{equation}\label{eq:hit}
 		\inf_
 		{y \notin T(x, |x|^\delta)} P_y[H_
 		{T(x, |x|^{(1 - \varepsilon)\delta})} = \infty] \geq c \delta\varepsilon.
 	\end{equation}
 \end{prop}
 
 \begin{proof} We abbreviate 
 \begin{equation*}
 T = T(x),  \quad T^2 = T(x, |x|^{\delta(1 - \varepsilon)}), \quad T^3 
 = T(x, |x|^{\delta}) 
 \end{equation*}
 throughout this proof. 
 Let us start with a few reduction steps. By a straightforward application of the strong Markov 
 property at time $H_{T^3}$ it is enough to show \eqref{eq:hit} for $y \in 
 \partial^{\text{out}} T^3$. {For $|x| \ge C(\delta)$, we can write $\partial^{\text{out}} T^3 
 	\subset \mathbb{S} \cup \mathbb{L}$,} 
 	where, 
 	writing $R= |x|^{\delta}$ 
 	and 
 	$\ell=\{te: t \in \mathbb{R}\}$ with $e = x/|x|$ and tacitly embedding 
 	$\Z^d\subset \R^d$ in writing expressions such as $d(z, \ell)$ below, we set
 	\begin{equation} \label{eq:short-long}
 		\begin{split}
 			& {\mathbb{L} \stackrel{\text{def.}}{=}  
 				\big\{ z \in T(x, 4
 				R) \setminus T^3: d(z, \ell) \ge \tfrac
 				{R}{4} \big\},} \\
 			& {\mathbb{S}\stackrel{\text{def.}}{=}  \partial^{\text{out}} T^3 \setminus \mathbb L}.
 		\end{split}
 	\end{equation}
 	One can reduce the case $y \in \mathbb{S}$ to the case $y \in \mathbb{L}$, as explained at the end of the proof. We now focus on $y \in \mathbb{L}$. For such $y$, we will in fact show a stronger statement, with $T^2$ in \eqref{eq:hit} replaced by a certain enlargement $\widetilde{T}^2 \supset T^2$, which we now introduce. This enlargement will later have a `uniformizing' effect on the martingale we consider, cf.~\eqref{eq:hit-1} and \eqref{eq:mart-2} below.
 	
 	\begin{figure}[htb!]
 		 \centering
 		\includegraphics[width=\linewidth]{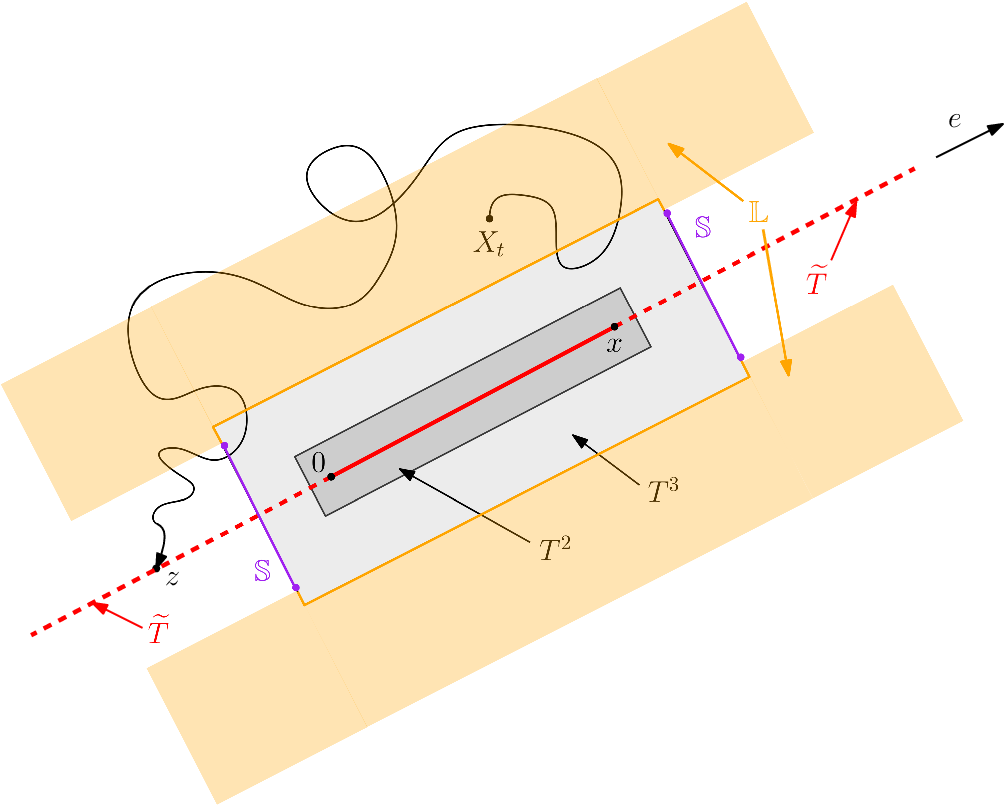}
 		\caption{{\bf A configuration of cylinders in a generic direction $\bm{e = {x}/{|x|}}$.} We illustrate a sample path contributing to the martingale $M_t$ introduced in \eqref{eq:hit-1}. The point $z$ varies over all of $\widetilde{T}$, which includes the two `stabilizing' rods (dashed). These are used to achieve the desired uniformity for the lower bound in \eqref{eq:mart-2}. In reality, all lines and rectangles drawn are replaced by lattice approximations, cf.~\eqref{eq:T_x}-\eqref{eq:T_i}.}
 		\label{fig:mart}
 	\end{figure}
 	
 	Recalling the notation $[z]$ from around \eqref{eq:T_x} and for $K \subset \mathbb{Z}^d$, let
 	\begin{equation} \label{eq:K-tilde}
 		\widetilde{K}\stackrel{\text{def.}}{=} K \cup \big\{ [j
 		e] : j \in \mathbb{Z} \cap \big( [-|x| , 0] \cup [ |x| , 2|x|] \big) \big\}.
 	\end{equation}
 	In words, $\widetilde{K}$ is obtained from $K$ by adding two `rods' of length roughly $|x|$ each, extending from near $0$ and $x$ in opposite directions parallel to $e$.
 	Abbreviating $T(x) \equiv T$ for the set from \eqref{eq:T_x} in the sequel,  let
 	\begin{equation}
 		\label{eq:hit-1}
 		M_t = \sum_{z \in \widetilde{T}} g(X_t,z),
 	\end{equation}
 	and define the stopping time $\tau= H_{\widetilde{T}^2} \wedge T_U= T_{U\setminus \widetilde{T}^2} $ with $U = B(0,10^3|x|)$.

 	Since $g(\cdot, z)$ is harmonic outside $z$ and using the fact that $\widetilde{T} \subset \widetilde{T}^2$ it readily follows that $(M_{t\wedge \tau})_{t \geq 0}$ is a martingale under $P_y$ for all $y \in \mathbb{L}$. We will need to separately consider the extremities $\text{Ext}(\widetilde{T}) \subset \widetilde{T}$, defined as follows. Let $x_{\pm} \in \widetilde{T}$ where $x_+= [2|x| e]= [2x]$ and $x_-= [-|x| e]=[-x]$ and
 	\begin{equation}
 		\label{eq:hit-2}
 		\text{Ext}(\widetilde{T}) \stackrel{\text{def.}}{=} \{ z \in \widetilde{T}: d(z,x_+)\wedge d(z,x_-) \leq |x|/2\}.
 	\end{equation}
 	With these definitions, the following hold.
 	\begin{lemma}[$d = 3$]
 		\label{C:mart} Let 
 		\begin{equation}\label{eq:c_2}
 			\Cl[c]{c:Green}=\lim_{|x| \to \infty} |x|\cdot { g(0,x)} \ (=\textstyle\frac3{2\pi})
 		\end{equation}
(see, e.g.~\cite[Theorem~1.5.4]{Law91}). For $|x| \geq C(\delta, \varepsilon)$, one has:
\begin{align}
 			&\label{eq:mart-1} \text{for all $z \in \mathbb{L}$: } E_z[M_0] \leq 2 \Cr{c:Green} 
 			(1 - (1 - 0.1\varepsilon)\delta) \log |x|, \\
 			&\label{eq:mart-2} \text{for all $z \in \widetilde{T}^2 \setminus \textnormal{Ext}(\widetilde{T})$: } E_z[M_0] \geq 2\Cr{c:Green} 
 			(1 - (1 - 0.9\varepsilon)\delta)
 			\log |x|. 
 		\end{align}
 	\end{lemma}
 	The proof of both \eqref{eq:mart-1} and \eqref{eq:mart-2} requires slight care owing to 
 	discrete effects. To avoid disrupting the flow of reading the proof is postponed to the end of this section.
 	Suppose now that Lemma~\ref{C:mart} holds. We proceed to show \eqref{eq:hit} for $y \in \mathbb{L}$. Since $|M_{t\wedge \tau}| \leq C$ $P_y$-a.s., the optional stopping theorem applies and yields, upon
 	neglecting the (positive) contributions stemming from both the event 
 	$\{H_{\widetilde{T}^2} > T_U\}$ and otherwise from the case where 
 	$X_{\tau}=X_{H_{\widetilde{T}^2}}$ belongs to $ \text{Ext}(\widetilde{T})$, that for all $|x| 
 	\geq C(\delta, \varepsilon)$ and $ y\in \mathbb{L}$,
 	\begin{multline}\label{eq:hit-5}
 		2\Cr{c:Green} 
 		(1 - (1 - 0.1\varepsilon)\delta)\log|x| \stackrel{\eqref{eq:mart-1}}{\geq} E_y[M_0] \\= E_{y}[M_{\tau}] \stackrel{\eqref{eq:mart-2}}{\geq} 
 		P_y\big[H_{\widetilde{T}^2} \leq T_U, X_{\tau} \notin \text{Ext}(\widetilde{T})\big] \cdot 2 \Cr{c:Green} 
 		(1 - (1 - 0.9\varepsilon)\delta) \log|x|. 
 	\end{multline}
 	By definition of $\widetilde{T}$ and $\text{Ext}(\widetilde{T})$, see \eqref{eq:K-tilde} and \eqref{eq:hit-2}, see also \eqref{eq:T_i} one has that $d(T(x, 4R), \text{Ext}(\widetilde{T})) 
 	\geq \frac{|x|}{4}$ when 
 	$|x| \geq C(\delta, \varepsilon)$. By \cite[Lemma 2.2 and Remark 
 	2.3]{gosrodsev2021radius} (or using \eqref{eq:cap-1} below), one knows that 
 	$\text{cap}\big(\text{Ext}(\widetilde{T})\big) \leq C\frac{|x|}{\log |x|}$. By 
 	\eqref{eq:last-exit}, it thus follows that
 	\begin{equation}
 		\label{eq:hit-6}
 		\sup_{y \in T(x, 4R)} P_y\big[H_{\text{Ext}(\widetilde{T})}< \infty\big] \leq C( \log|x|)^{-1} \to 0 \text{ as $|x| \to\infty$.}
 	\end{equation}
Using \eqref{eq:hit-6} one deduces in particular that $P_y\big[H_{\widetilde{T}^2} \leq 
T_U, X_{\tau} \in \text{Ext}(\widetilde{T})\big] \leq 0.2\Cr{c:Green}\delta\varepsilon$ 
for $y \in \mathbb{L}$ whenever $|x| \geq C(\delta, \varepsilon)$. Using this together with 
\eqref{eq:hit-5} and the fact that $T_2 \subset \widetilde{T}_2$ by \eqref{eq:hit-1}, it follows 
that 
\begin{equation}
\label{eq:hit-7}
P_y[H_{{T}^2} \leq T_U] \leq P_y[H_{\widetilde{T}^2} \leq T_U] \leq 1- 2 \Cr{c:Green}  
\Big(\frac{0.8\delta\varepsilon}{1 - (1 - 0.9\varepsilon)\delta}- 0.1 \delta \varepsilon\Big),
\end{equation}
for all $y \in \mathbb{L}$ and $|x| \geq C(\delta, \varepsilon)$. With a similar calculation as in 
\eqref{eq:hit-6}, one finds that $\sup_{z \in \Z^d} P_z[ T_U< H_{{T}^2}< \infty]\to 0$ as $|x| 
\to \infty$. Combining with \eqref{eq:hit-7}, one deduces the bound in \eqref{eq:hit} uniformly 
in $y \in \mathbb{L}$. 
 	
 	To complete the proof of \eqref{eq:hit} it thus remains to treat the case $y \in \mathbb{S}$ in \eqref{eq:hit}. To deal with this, we now argue that
 	\begin{equation}
 		\label{eq:hit-8}
 		\inf_{y \in \mathbb{S}}P_y[H_{\mathbb{L}} < H_{T^2}] \geq c.
 	\end{equation}
 	Once \eqref{eq:hit-8} is shown, \eqref{eq:hit} follows by applying the strong Markov property at time $H_{\mathbb{L}}$ and combining \eqref{eq:hit-8} with the lower bound 
 	on $P_y[H_{T^2}=\infty]$ for $y \in \mathbb{L}$ already derived. We show \eqref{eq:hit-8} 
 	using a chaining argument. Let $y \in \mathbb{S}$. By definition, see \eqref{eq:short-long} 
 	and \eqref{eq:T_i} for any such $y$ we can find  a finite number of boxes 
 	{$B_i=B(x_i, \frac{R}{100})$}, $1 \leq i \leq n$ (with $n$ uniform in $y$) such 
 	that $B_1 \ni y$, $B_{i+1}$ and $B_i$ have a non-empty overlap containing a translate of $B(0, \frac{R}{300})$ for all $1 \leq i < n$, 
 	$B_n \subset \mathbb L$ and lastly $B(x_i, \frac{R}{10}) \cap T_2 = \emptyset$ for all $1\leq i \leq n$. Then, using 
 	\cite[(A.4)]{PRS23} with the choice $\delta=1$, one readily infers that $\inf_{z \in 
 		B_i}P_z[H_{B_{i+1}}< H_{T^2}] \geq c$ for all $1 \leq i < n$.   Applying the strong Markov property repeatedly, \eqref{eq:hit-8} follows as 
 	$B_n \subset \mathbb L$.
 \end{proof}
We now give the proof of the estimates  \eqref{eq:mart-1} and \eqref{eq:mart-2}, for which 
the stabilizing effect of the extension $\widetilde{T}$ comes into effect.
\begin{proof}[Proof of Lemma~\ref{C:mart}]
Recall $M_t$ from \eqref{eq:hit-1}, which involves summation along the oblique line $\widetilde{T}$, see \eqref{eq:T_x} regarding $T=T(x)$ and \eqref{eq:K-tilde} regarding its extension $\widetilde{T} \supset T$. As before, with $|\cdot|$ denoting the Euclidean norm, let $e=\frac x{|x|}$ for $x \neq 0$, and viewing $\Z^d \subset \R^d$, consider the line $\ell=\{te: t \in \mathbb{R}\}$.
	Given $z \notin \ell$, define ${y}_z \in \ell $ as the point minimizing the distance to $z$, so that, with $\langle \cdot, \cdot\rangle$ denoting the standard inner product,
	\begin{equation}
		\label{eq:app1}
		\langle z- {y}_z, e \rangle=0
	\end{equation}
	and let $j_z \in \mathbb{Z}$ be such that $|je- y_z|$ is minimized when $j= j_z$. In particular, it follows that
	\begin{equation}
		\label{eq:app1-bis}
		|j_ze- y_z| \leq C.
	\end{equation}
	From this, one obtains that for all $y \in \widetilde{T}$, i.e.~for all $y$ of the form $y=[je]$ for some $j \in \mathbb{Z} \cap [-|x|,2|x|]$, and $z \in \mathbb{Z}^d$, abbreviating $d_z= d(z,\ell)$ (recall that $d(\cdot,\cdot)$ refers to the Euclidean distance between sets), whence $d_z=|z- y_z|$, that
	\begin{multline} \label{eq:app3}
		|y-z| =  |[je]-z| \geq |je-z|- C \\
		\stackrel{(\ast)}{=}  \sqrt{ | je- y_z|^2 + d_z^2} - C
		\geq d_z  \sqrt{ 1+ \frac{ ( (|j- j_z| -\Cl{C:gf-estim22}) \vee 0)^2}{d_z^2}}  - C,
	\end{multline}
	where we used in deriving $(\ast)$ that $y_z - e$ is collinear with $e$, whence $\langle je- 
	y_z, z- {y}_z \rangle =0$, and in the last step that $| je- y_z| \geq |j- j_z|- |j_z e- y_z| \geq  
	|j- j_z|- C$ using \eqref{eq:app1-bis}. Now if $z\in \mathbb{L}$ as in \eqref{eq:mart-1}, then 
	in view of \eqref{eq:short-long} one has that $j_z \in [-|x|/2, 3|x|/2 \rceil]$ whenever $|x| 
	\geq C(\delta, \varepsilon)$. For such $z$, let $I_k= \{ j \in \mathbb{Z}:  |j- j_z| 
	-\Cr{C:gf-estim22} \in [k d_z , (k+1) d_z) \}$ and $K = 2|x|/d_z$, so that
	\begin{equation} \label{eq:app5}
		\{ j \in \mathbb{Z}: [je] \in \widetilde{T}\} \setminus [j_z -2\Cr{C:gf-estim}, j_z+2 \Cr{C:gf-estim}] \subset \bigcup_{k=0 }^K I_k.
	\end{equation}
	It follows that for $z \in \mathbb{L}$ and $|x| \geq C'(\delta, \varepsilon)$, recalling $M_t$ 
	from \eqref{eq:hit-1}, one has
\begin{multline} \label{eq:app4}
(1+ \textstyle \frac{\varepsilon}{10^3})^{-1} E_z[M_0] \displaystyle\stackrel{\eqref{eq:c_2}}{\leq}  
\Cr{c:Green} \sum_{y \in \widetilde{T}} \frac{1}{|z-y|} \\\stackrel{\eqref{eq:app5}}{\leq} C+ 
\Cr{c:Green} \sum_{k=0}^K \sum_{y \in \widetilde{T} } \frac{ 1{\{ y = [je] \text{ for a } j \in I_k\}} 
}{ |z-y|} \stackrel{\eqref{eq:app3}}{\leq} C'+ \Cr{c:Green} d_z^{-1} \sum_{k=1}^K |I_k| 
\frac{1}{\sqrt{1+ k^2}- \frac C{d_z}}.
\end{multline}
Now, using the fact that 
$c|x|^{\delta} \leq d_z = d(z,\ell) \leq C|x|^{\delta}$ for all $x$ and $z \in \mathbb{L}$ one 
infers that the last fraction in \eqref{eq:app4} is at most 
$(1+ \frac{\varepsilon}{10^3}) k^{-1}$ for $|x| \geq C(\delta)$, that 
$K \leq C |x|^{1-\delta}$ and that $|I_k|=2 d_z$ for any such $k$. Substituting these bounds 
into \eqref{eq:app4}, the bound \eqref{eq:mart-1} readily follows upon performing the 
harmonic sum and using that $\sum_{1 \leq k \leq K}\frac1{k} \leq 1+ \log K$.

	To obtain \eqref{eq:mart-2}, one proceeds similarly with a few modifications, which we now highlight. It is here that the extension $\widetilde{T}$ bears its fruits over $T$ and allows a not too wasteful estimate (in particular, one correctly producing the pre-factor $2$ appearing on the right-hand side of \eqref{eq:mart-2}). One readily derives a companion bound to \eqref{eq:app3}, yielding that for all $z \in \mathbb{Z}^d$ and $j \in \mathbb{Z}$,
	\begin{equation}\label{eq:app6}
		|[je]-z| \leq d_z  \sqrt{ 1+ \frac{ ( |j- j_z| +\Cl{C:gf-estim})^2}{d_z^2}}  + C.
	\end{equation}
	The fact that $j_z$ as defined above \eqref{eq:app1-bis} continues to range in $[-|x|/2, 3|x|/2 \rceil]$ when $z \in \widetilde{T}^2 \setminus \textnormal{Ext}(\widetilde{T}) $ as in \eqref{eq:mart-2} follows readily using the definition of $\textnormal{Ext}(\widetilde{T})$ in \eqref{eq:hit-2}. Thus, in view of the definition of $\widetilde{T}$ in   \eqref{eq:K-tilde}, if one sets $K' = c |x|/d_z$ for sufficiently small $c$ and defines $I_k'= \{ j \in \mathbb{Z}:  |j- j_z|  \in [k d_z , (k+1) d_z) \}$, then one obtains the inclusion  
	\begin{equation}\label{eq:app7}
		\{ j \in \mathbb{Z}: [je] \in \widetilde{T}\} \supset \bigcup_{k =1}^{K'} I_k'.
	\end{equation}
	Equipped with \eqref{eq:app6} and \eqref{eq:app7}, now playing the role of \eqref{eq:app3} and \eqref{eq:app5}, respectively, one readily performs a computation akin to \eqref{eq:app4}, using among others the complementary bound stemming from \eqref{eq:c_2} as well as the fact that the sets $I_k'$ are disjoint, to conclude \eqref{eq:mart-2}.
\end{proof}
	%
 \section{Generalized lower bound}\label{sec:gen}
 
 We now state a general lower bound on a certain class of events involving simultaneous connection and disconnection events in the vacant set which is valid across the {\em full} 
 range of parameter values. The precise statement is  given in Theorem~\ref{T:LB-gen} below. The event of interest in Theorem~\ref{T:LB-gen} 
 (cf.~\eqref{eq:LB-final-3}) has three distinguishing features, which together  allow for various interesting consequences. First, it is quantitative in the sense that 
 connection and disconnection are implemented in specific (tubular) regions  introduced below.   Second, the event allows a small amount of unfavorable sprinkling 
 between the levels involved in the connection and disconnection events (parametrized by $\eta > 0$ below). Third, it leaves the flexibility for a good event $G^u(x)$, 
 see \eqref{eq:LB-final-2}, which in practice allows for local surgery around  $0$ and $x$: indeed, the event in Theorem~\ref{T:LB-gen} only connects their 
 $r_1$-neighborhhood, where $r_1$ is chosen below, see \eqref{eq:LB-final-1}. 
 
Theorem~\ref{T:LB-RI}, as well as various other interesting lower bounds (see for instance 
Remark~\ref{R:1-arm}), are then derived in the remainder of this section, using 
Theorem~\ref{T:LB-gen} as crucial ingredient. Theorem~\ref{T:LB-RI} follows rather 
straightforwardly once one implements a suitable surgery to connect $0$ and $x$ (without 
spoiling disconnection). 
 
We will work with some special cylindrical sets (recall \eqref{eq:T_x}) which will play a role in 
the sequel. To this end, let $x \in \Z^d \setminus \{0\}$
 (eventually this will correspond to the point $x$ appearing in the statement of 
 Theorem~\ref{T:LB-RI}) and $\delta \in (0,\frac16)$. Our construction involves the nested 
 sets 
 \begin{equation}
 \label{eq:T_i}
 \begin{split}
 &T^1 \subset \dots \subset T^6, \text{ where } \\
 &T^i=T(x,(|x| \vee 100)^{i \delta}), \, 1 \leq i \leq 5, \quad T^6= T(x,4(|x| \vee 100)^{5 \delta}), 
\end{split}
 \end{equation}
 which implicitly depend on $x$ and $\delta$. Let 
 \begin{equation}\label{eq:LB-final-1}
 r_0 =r_0(x,\delta) \stackrel{\text{def.}}{=} 100 \vee \frac{|x|^\delta}{100}, \quad r_1=r_0^{1/4}.
 \end{equation}
 In the sequel 
 $(\ell_y^u)_{y\in \Z^d, u>0}$ denotes the occupation time field of the 
 interlacement point process; for its explicit definition cf.~\eqref{eq:occ-time} below. In what 
 follows we consider, for $u>0$, $x \in \Z^d $, $\delta \in (0,1)$ and $\eta \in [0,1)$ (often kept 
 implicit in our notation),  a generic event $G^u$ of the form $G^u=\bigcap_{i \in I} 
 G^u_{B_i}$, where $(B_i : i \in I)$ is a finite sequence of boxes of radius $10 r_1$ with 
 centers in $T(x)$. It may well be $B_i=B_j$ for some $i \neq j$. We further assume that 
 $G_B=G_B^u$ can be expressed measurably in terms of (at most) two local occupation 
 fields $(\ell_y^v)_{y \in B}$ and $(\ell_y^{w})_{y \in B}$, where $v,w \in 
 \{u(1- \frac k4\eta \sqrt{\delta}): k=0,1,2,3,4\}
 $ with $v>w$, and that $G_B$ is either increasing in $(\ell_y^v)_{y \in B}$ and decreasing in $(\ell_y^w)_{y \in B}$, or simply monotone if $G_B$ is measurable in terms of the occupation time field at a single level only. 

 \begin{thm} \label{T:LB-gen}   Let $G^u$ be of the above form and define, for all $u \in (0, \infty)$, $\eta \in [0,1)$, $x \in \Z^d$ and $\delta \in (0,1)$, the event
  \begin{equation}
\label{eq:A-final}
 A^{u}(x, \delta) = \big\{ G^u, \, \lr{}{\mathcal V^u \cap T^1}{B(0,r_1)}{B(x,r_1)}, \, \nlr{}{\mathcal V^{u(1-\eta \sqrt{\delta})}}{T^4}{\partial T^5} \big\}.
\end{equation}
If $(B_i : i \in I)$ is a finite sequence of boxes of radius $10 r_1$ with 
 centers in $T(x)$ and the following is true:
 \begin{equation}\label{eq:LB-final-2}
 \lim_{|x| \to \infty} \,  |I| \cdot  \sup_{i}  \P[(G_{B_i}^{u})^c]=0, \text{ for all  $u \in \textstyle[2^{-1}{ u_*},2 u_*]$,}
 \end{equation}
 then for all $u \in (0, \infty)$, $0 \leq \eta \leq \Cl[c]{C:eta} $, one has when $d=3$ that
  \begin{equation}\label{eq:LB-final-3}
\liminf_{\delta \downarrow 0} \liminf_{|x | \to \infty}\frac{\log |x |}{|x|}\log \P \big[ A^{u}(x, \delta)  \big] 
\geq - \frac{\pi}{3}\big(\sqrt{u_*} - \sqrt{u} \big)^2 - C\eta^2.
 \end{equation}
 \end{thm}
 
 
 The proof of Theorem~\ref{T:LB-gen} appears at the end of Section~\ref{sec:tilt}. It will be seen there to quickly follow from two main intermediate results, Propositions~\ref{P:entropy} and~\ref{P:typical}, which are proved in subsequent sections.
 As a first consequence of Theorem~\ref{T:LB-gen}, we derive in the remainder of the present section the asserted lower bound on 
 the truncated two-point function $\tau_{u}^{{\rm tr}}$ from \eqref{def:truncated_twopt}. 
 The event $G^u$ will thereby allow for local surgeries (around $0$ and $x$) at affordable 
 \textit{multiplicative} cost. This will notably involve a sprinkled finite energy technique 
 recently developed in \cite[Section 3]{RI-II}.
 
\begin{proof}[Proof of Theorem~\ref{T:LB-RI}] 
We first prove \eqref{eq:LB_RI-1} for which we need to strengthen the connection 
between $B(0, r_1)$ and $B(x, r_1)$ in \eqref{eq:A-final} to a connection between $0$ and 
$x$ as required by \eqref{def:truncated_twopt}. One obvious approach would be to take an 
intersection with the event $\{B_0\cup B_x \subset \mathcal{V}^u\}$ and use the 
FKG-inequality. But the absence of monotonicity of the event appearing in \eqref{eq:A-final} 
makes it inapplicable in this case. 

To address the issue we make use of the event 
$G^u$. For $y \in \mathbb{Z}^d$, let $A_y \subset \widetilde{A}_y$ denote the annuli 
defined as $A_y =  B(y, 6{r_1}) \setminus B(y, 4{r_1})$ and $\widetilde{A}_y =  B(y, 7{r_1}) 
\setminus B(y, 3{r_1})$. By choice of $r_1$ in \eqref{eq:LB-final-1} and in view of 
\eqref{eq:T_i}, we have that $(\widetilde{A}_0 \cup \widetilde{A}_x) \subset T^1$, as will be 
needed for the event defined momentarily in \eqref{eq:LB-final-5} to satisfy the assumptions 
of Theorem~\ref{T:LB-gen}.We will eventually apply Theorem~\ref{T:LB-gen} with $v=  
\frac{u}{1-\eta\sqrt{\delta}} \, ( > u)$ in place of $u$ and $\eta \in (0, \Cr{C:eta})$. To this 
effect, let
 \begin{equation}\label{eq:LB-final-5}
G^v(x) \stackrel{\text{def.}}{=} \bigcap_{z \in \{0,x \}} G_z^1\cap  G_z^2 \cap  G_z^3
\end{equation}
 where, setting $v_2 = v(1-\frac{\eta \sqrt{\delta}}2)$ so that $u < v_2 < v$, we define
 	\begin{equation}
	\begin{split}\label{eq:fin_energy_good}
				&{{G}}_{z}^1
\stackrel{{\rm def.}}{=} \bigcap_{y, y' \in \mathcal I^{v_2} \, \cap \, A_z}\big\{  \lr{}{\mathcal I^{v} \, \cap\,\widetilde{A}_z}{y}{y'}  \big\}, \\ 				
&{{G}}_{z}^2 \stackrel{{\rm def.}}{=} \big\{ (\mathcal I^{v_2} \setminus\mathcal I^{u})\cap B(z,2r_1)\ne \emptyset\big\}, \mbox{ and }
{{G}}_{z}^3
\stackrel{{\rm def.}}{=} \bigcap_{y \in B(z, 9r_1)} \{\ell_y^{v} \le {r_1}\},
		\end{split}
\end{equation}
and $\ell^v_{\cdot}$ denote the (discrete) occupation time field of random interlacements at level $v$. As we now briefly elaborate, one readily checks that the event $G^v=G^v(x)$ is of the form required by Theorem~\ref{T:LB-gen}, for a choice of boxes $B_i$, $i \in I$, with $|I|=6$, each centered at $0$ or $x$. For instance, the event $G_0^1$ in \eqref{eq:fin_energy_good} can be represented as $G_B^v$ with $B=B(0,10r_1) \, (\subset T^1)$ and expressed as a function increasing in $(\ell^{v}_x)_{x\in B}$ and decreasing in $(\ell^{v_2}_x)_{x\in B}$. The other events appearing in \eqref{eq:fin_energy_good} are dealt with similarly.

 As $r_1=r_1(x) \to \infty$ as $|x| \to \infty$, the fact that \eqref{eq:LB-final-2} holds  with $v$ in place of $u$ (for any $v>0$) follows by standard arguments, as we now briefly explain. To deal with $G_z^1$, one  uses \cite[Theorem 5.1]{DPR22}. As to \eqref{eq:LB-final-2} for $G_z^2$, one simply uses that $\mathcal{I}^{v_2} \setminus \mathcal{I}^u\neq \emptyset$ holds almost surely. Lastly, observing that $\ell_y^v$ has the same law as $\sum_{i=1}^N \zeta_i$, where $N $ is a Poisson variable with mean $v/g$, $g=g(0,0)$, and $\zeta_i$ are i.i.d.~geometric variables starting at $1$ with parameter $p=P_0[\widetilde{H}_0< \infty]= 1- 1/g$, independent of $N$, and applying a union bound, one obtains that $ \P[(G_z^3)^c] \leq  \sum_{y \in B(z, 9r_1)} \P[ \ell_y^{v} \geq {r_1}] \leq C r_1^d e^{-c(v)r_1}\to 0$ as $r_1 \to \infty$. Thus, all in all, \eqref{eq:LB-final-3} yields for $d=3$ that
\begin{multline}\label{eq:LB-final-6}
\liminf_{\delta \downarrow 0} \liminf_{|x | \to \infty}\frac{\log |x |}{|x|}\log \P \big[G^v(x), \, \lr{}{\mathcal V^v \cap T^1}{B(0,r_1)}{B(x,r_1)}, \, \nlr{}{\mathcal V^{u}}{T^4}{\partial T^5} \big] \\
\geq - \frac{\pi}{3}\big(\sqrt{u_*} - \sqrt{u}\big)^2 - C\eta^2,
 \end{multline}
 for all sufficiently small $\eta>0$, with $G^v$ as in \eqref{eq:LB-final-5}. 
 
 As we now explain, the event $G^v$ allows to obtain the bound
 \begin{multline}\label{eq:LB-final-6}
\big(  \tau_u^{{\rm tr}}(x) \geq \big) \   \P \big[ \lr{}{\mathcal V^u \cap T^1}{0}{x}, \, 
\nlr{}{\mathcal V^{u}}{T^4}{\partial T^5} \big]\\\geq e^{-Cr_1^{2d}}  \P \big[G^v(x), \, 
\lr{}{\mathcal V^v \cap T^1}{B(0,r_1)}{B(x,r_1)}, \, \nlr{}{\mathcal V^{u}}{T^4}{\partial T^5} 
\big],
 \end{multline}
 valid for all $d \geq 3$, $\delta \in (0,c)$ and $\eta \in (0,c(u))$. Once \eqref{eq:LB-final-6} 
 is shown, taking logarithms, using \eqref{eq:LB-final-6} and \eqref{eq:LB-final-1}, the 
 desired lower bound in \eqref{eq:LB_RI-1} for $u < u_*$ follows upon taking successively 
 the limits $|x | \to \infty$ (for fixed $\delta< \frac2{d}$, so that $r_1^{2d}\leq C |x|^{d\delta/2}= C|x|^{1-c}$ for some $c=c(\delta) >0$), then $\delta \downarrow 0$ and 
 finally $\eta \downarrow 0$. 

It remains to argue that \eqref{eq:LB-final-6} holds. The events $G_z^i$ defined in \eqref{eq:fin_energy_good} correspond to those appearing in \cite[(3.8)]{RI-II} for the choices $B=B(z,r_1)$ and $u_1=u_2=u_3=v$, $2\delta_1=\delta_2= \eta \sqrt{\delta} v$. In particular, the intersection $G_z^1\cap G_z^2\cap G_z^3$ implies the event referred to as $\widetilde{F}_B$ therein (see for instance \cite[(3.9)]{RI-II} or \cite[Remark 3.2]{RI-II}), on which the conclusions of \cite[Proposition 3.1]{RI-II} hold for $r=r_1$ and with $(v,u)$ above corresponding to $(u,u-\delta)$ therein. In particular, this entails the following.
Let $\widehat{B}=B(z,8r_1)$ and $\omega_{\widehat{B}}^-$ denote the point measure obtained from $\omega=\sum_i \delta_{(w_i^*, u_i)}$, the canonical interlacement point measure, by removing from each trajectory $w_i^*$ intersecting $\widehat{B}$ all excursions between $\widehat{B}$ and $\partial_{\textnormal{out}}\widehat{B}$. Then abbreviating by $\mathcal{F}=\sigma(\omega_{\widehat{B}}^-, \mathcal{I}^v \cap \widehat{B})$, one has by \cite[(3.2)]{RI-II} that $\widetilde{F}_B \in \mathcal{F}$ and moreover that
\begin{equation}\label{eq:LB-final-7}
 \P\big[B\subset \mathcal V^{u } \, \big| \, \mathcal{F}\,\big]1_{\widetilde{F}_B} \geq  e^{-C r^{2d}_1}, \quad B=B(z,r_1), \, z\in \{0,x\}. 
 \end{equation}
 Observe now that each of the events $G_x^1$, $G_x^2$, $G_x^3$, $\widetilde{F}_B$, $\{\lr{}{\mathcal V^v \cap T^1}{B(0,r_1)}{B(x,r_1)} \}$ and $\{ \lr{}{\mathcal V^{u}}{T^4}{\partial T^5}\}$ is $\mathcal{F}$-measurable: indeed, with regards to $G_x^i$, one uses that $100 r_1 \leq |x|$. With regards to the two connection events, one uses that $\omega_{\widehat{B}}^-$ contains the information about all the configurations $\mathcal{V}^v\setminus \widehat{B}$ and  $\mathcal{V}^u\setminus \widehat{B}$, and that $(T^5 \setminus T^4) \cap \widehat{B} =\emptyset$.
 Thus, returning to the probability on the right-hand side of \eqref{eq:LB-final-6}, using for fixed $z$ (say $z=0$) the inclusion $G_z^1\cap G_z^2\cap G_z^3 \subset \widetilde{F}_B$, applying \eqref{eq:LB-final-7} allows to enforce the event $\{B(0,r_1) \subset \mathcal{V}^u\}$ at the cost of a multiplicative factor $e^{C {r}^{2d}_1}$. Repeating the procedure for $z=x$ (which now also includes the event $\{B(0,r_1) \subset \mathcal{V}^u\} \in \mathcal{F}$, where the latter follows using again that $100 r_1 \leq |x|$) yields the desired connection between $0$ and $x$ in $\mathcal{V}^u$, and \eqref{eq:LB-final-6} follows. This completes the verification of \eqref{eq:LB_RI-1}.

We now prove \eqref{eq:LB_RI-2}, which is far simpler. Let $\ell \subset \Z^d$ denote a 
connected set of minimal cardinality intersecting both $0$ and $x$. Thus $|x| \leq |\ell| \leq 
C|x|$, where $|\ell|$ denotes the cardinality of $\ell$. Defining $\Sigma = \partial B(\ell, 1)$, where $B(\ell, 1)=\bigcup_{y \in \ell} B(y,1)$, we note that any unbounded path intersecting $\ell$ must also intersect $\Sigma$. Let 
$\mathcal{N}^u$ denote the number of trajectories of the interlacement at level $u$ 
intersecting $\Sigma$, a Poisson variable with mean $u \cdot \text{cap}(\Sigma)$. By 
subadditivity of the capacity, one has that $ \text{cap}(\Sigma) \leq  |\Sigma|  \leq C|x|$, 
whence
\begin{equation}\label{eq:LBdgeq4}
\tau_{u}^{{\rm tr}}(x) \geq \P[ \ell \subset \mathcal{V}^u, \, \Sigma \subset \mathcal{I}^u ] \geq c u  e^{-Cu |x|}\P\big[ \ell \subset \mathcal{V}^u, \, \Sigma \subset \mathcal{I}^u \,  \big|\, \mathcal{N}^u=1 \big],
\end{equation}
where the first inequality is an inclusion by construction of $\ell$ and on account of the above observation on $\Sigma$. Moreover,  in bounding $\P[\mathcal{N}^u=1]$ from below to obtain the second inequality, we also used that $\text{cap}(\Sigma) \geq \text{cap}(\{0\}) \geq c$. Now, the conditional probability on the right-hand side of \eqref{eq:LBdgeq4} is bounded from below by
$$
\inf_{z \in \Sigma} P_z[\text{range}(X) \supset \Sigma, \, \text{range}(X) \cap \ell =\emptyset ],
$$ 
where $X$ is the simple random walk starting from $z$ under $P_z$. We claim that the latter probability is bounded from below by $e^{-C|x|}$, which, if true and when substituted in \eqref{eq:LBdgeq4}, completes the proof. The former can be seen as follows. Fix a deterministic path $\gamma$ contained in $\Sigma$ starting in $z\in \Sigma$ whose range covers $\Sigma$ (which is a connected set). One can choose $\gamma$ in such a way that its length is bounded by $C|x|$, uniformly in $z$. Forcing $X$ to follow this path in its first steps ensures that $\text{range}(X) \supset \Sigma$, at a cost bounded from below by $(2d)^{-C|x|}$. Upon completing this requirement, one forces $X$ at a similar cost to move away from $\Sigma$ without intersecting $\ell$ following another deterministic path until reaching distance $|x|$ from $\Sigma$. For $z'$ at distance $|x|$ from $\Sigma$, denoting by $H_\Sigma$ the entrance time in $\Sigma$ one readily infers that, as $|x| \to \infty$, $$P_{z'}[H_\Sigma < \infty] \leq C|x|^{2-d} \text{cap}(\Sigma) \leq C' |x|^{3-d} \to 0.$$ The assertion now follows by applying the Markov property for $X$ and combining the various ingredients.
  \end{proof}
  
  \begin{remark}[one-arm estimates]
  \label{R:1-arm}
  \begin{enumerate}[label={\arabic*)}]
  \item \label{R:1-arm-real}
 By picking $x= N e_1$, one immediately infers from Theorem~\ref{T:LB-RI} similar bounds 
 for the (truncated) one-arm probability, by which,
 \begin{align}\label{eq:LB_RI-1'}
&\text{for }d=3: \quad \liminf_{N \to \infty}\frac{\log N}{N}\log  \P[\lr{}{\mathcal V^u}{0}{\partial B_N}, \nlr{}{\mathcal V^u}{0}{\infty}] \geq  -\frac{\pi}{3}(\sqrt{u} - 
\sqrt{u_*})^2
\end{align}
for all $u \in (0, \infty)$, whereas
\begin{align}\label{eq:LB_RI-2'}
&\text{for }d\geq 4: \quad \liminf_{N \to \infty} \frac1N\log  \P[\lr{}{\mathcal V^u}{0}{\partial B_N}, \nlr{}{\mathcal V^u}{0}{\infty}] \geq  -C(u). 
\end{align}
for all $u \ne u_\ast$. See \cite[Theorem 6.1 and Theorem 6.3]{GRS24.1} for matching upper bounds in the sub- and super-critical phase, respectively; see also \cite[Theorem 1.3]{prevost2023passage} in the sub-critical phase, including more general `first-passage percolation' distances.
\item \label{R:1-arm-quant}  The following lower bound, which both i) localizes the 
disconnection from $\infty$, and  ii) allows for a small unfavorable sprinkling can be useful in 
applications and is of independent interest (see below \eqref{eq:LU-intro}). 
Let $\delta \in (0,1)$ and $T_{\delta}= [-N^{\delta}, N+ N^{\delta}]\times 
[-N^{\delta}, N^{\delta}]^{d-1}$. Then for all $u< u_*$ and $0 \leq \eta < \Cr{C:eta}$, when 
$d=3$, 
one has that
\begin{multline}\label{eq:LB_RI-1'-quant}
\liminf_{\delta \downarrow 0} \liminf_{N \to \infty}\frac{\log N}{N}\log  \P\Big[\lr{}{\mathcal 
V^u \cap T_{\delta}}{0}{Ne_1}, \nlr{}{\mathcal V^{u(1-\eta \sqrt{\delta})}}{T_{2\delta}}{\partial 
T_{3\delta}}\Big] \\\geq - \frac{\pi}{3}\big(\sqrt{u_*} - \sqrt{u}\big)^2 - C\eta^2.
\end{multline}
The estimate \eqref{eq:LB_RI-1'-quant} follows by a slight modification of the above proof of 
Theorem~\ref{T:LB-RI}, which we now explain. Inspecting that proof, one chooses $x=N 
e_1$ and replaces the occurrences of $v_2$ and $u$ when defining the events $G_z^i$ in 
\eqref{eq:fin_energy_good} by $v_2 \coloneqq v(1-\frac{\eta \sqrt{\delta}}4)$ and $v_3 
\coloneqq v(1-\frac{\eta \sqrt{\delta}}2)$, respectively. The resulting events are still of the 
form required above Theorem~\ref{T:LB-gen}, and \eqref{eq:LB-final-2} continues to hold. 
Following  \eqref{eq:LB-final-6} and the subsequent arguments, these choices effectively 
allow to replace $u$ by $v_3$ in \eqref{eq:LB-final-7}. Together with the connection event 
in $\mathcal V^v \cap T^1$ appearing on the right of \eqref{eq:LB-final-6}, this leads to a 
connection between $0$ and $Ne_1$ at level $v_3 > u$, where $u$ is the level of 
disconnection in \eqref{eq:LB-final-6}. The lower bound \eqref{eq:LB_RI-1'-quant} then 
readily follows after a straightforward reparametrization. 
 
 When $ d\geq 4$, \eqref{eq:LB_RI-1'-quant} remains true upon making the following amendments. The factor $\log N$ appearing on the left-hand side is removed and the sets $T_{\delta}, T_{2\delta}$ and $T_{3\delta}$, are replaced by $[0,N]\times \{ 0\}^{d-1}, [0,N]\times \{ 0\}^{d-1}$ and $\partial ( [0,N]\times [-1,1]^{d-1})$, respectively. The bound on the right-hand side is then replaced by $-C(\eta,u)$; this bound follows directly by inspection of the proof of \eqref{eq:LB_RI-2}, for the choice $x=Ne_1$ (see the argument around \eqref{eq:LBdgeq4}).
\end{enumerate}
  \end{remark}

\section{The tilted measure $\widetilde{\mathbb{P}}_{f}$}
\label{sec:tilt}
The remaining sections are devoted to the proof of Theorem~\ref{T:LB-gen}. In the present section we introduce a certain tilt of the measure which will effectively render the event 
appearing in Theorem~\ref{T:LB-gen} typical. The tilted interlacement measure $\widetilde{\mathbb{P}}_{f}$ is introduced in \eqref{eq:PtildeAllLvlsdefn} below. It depends on a profile function $f$ which corresponds to a carefully chosen spatial modulation of the 
intensity (see Figure~\ref{fig:profile}). The key features of the measure $\widetilde{\mathbb{P}}_{f}$ are given by Propositions~\ref{P:entropy} and~\ref{P:typical} below, from which Theorem~\ref{T:LB-gen} is deduced at the end of the present section.

We start by introducing a minimal amount of notation that will be needed from here onwards. 
Recall the tubes $T^i=T^i(x)$ from \eqref{eq:T_i}, which depend on a choice of point $x \in 
\Z^d \setminus\{0\}$ and $\delta \in (0,1)$. In view of \eqref{eq:T_i}, one has the inclusions 
 \begin{equation}
 \label{eq:T_i-U}
 \begin{split}
 &(T^1 \subset \dots \subset) T^6 \subset U, \text{ where } U = B(0,10^3|x|).
\end{split}
 \end{equation}
We now introduce, for $\varepsilon \in (0,\frac{u_*}{10} \wedge 1)$ and $0<  v \leq u$, the 
profile function $f: \Z^d \to \R_+$ defined as 
\begin{align}
	{f}(y) \equiv {f}^{v,u; \varepsilon}(y) = 1 + \Big( \sqrt{\frac{u_{*}+\varepsilon}{v}} -1\Big) h_2(y) - \Big( \sqrt{\frac{u_{*}+\varepsilon}{v}} - \sqrt{\frac{u_{*}-\varepsilon}{u}} \Big)h_1(y), 
	\label{eq:f1hatdef}\end{align}
where $h_1(y) = P_y[H_{T^2} < T_{T^3}]$ and $h_2(y) = P_y[H_{T^6} < T_{U}]$ (see 
above \eqref{eq:g_U} for notations). Let us note that
\begin{align}
{f}(y) = \begin{cases} \sqrt{\frac{u_{*} - \varepsilon}{u}}, & x \in T^2
 \\ \sqrt{\frac{u_{*}+\varepsilon}{v}}, &x\in T^6\setminus T^3 \\ 1, &x\notin U \end{cases} \label{eq:chosenf-2} 
\end{align}
and is interpolated harmonically in between these regions; see Figure~\ref{fig:profile} and 
revisit the discussion around ``harmonic humpback'' in the introduction.

In light of this, we can now define a tilted interlacement measure $\widetilde{\P}_f$ for $f$ as in \eqref{eq:f1hatdef}, which 
we introduce next. The following construction is essentially the same as in \cite[Section 
2]{zbMATH06257634}, to which we frequently refer, except that we retain information on 
the labels $u$ for reasons related to the function $f$ and the non-monotonicity of the 
event of interest in the super-critical phase. Let ${W}_+$, resp.~${W}$ denote the space of 
infinite, resp.~bi-infinite continuous-time transient $\Z^d$-valued trajectories, with finitely 
many jumps in bounded intervals of times. We write $X_t$, $t \in \mathbb{R}$ (resp.~$t\geq 
0$) for the canonical coordinates on $W$ (resp.~$W_+$). Identifying trajectories $w,w' \in 
{W}$ using the equivalence relation ${w} \sim {w}'$ if $w(\cdot) = w'(\cdot + t)$ for some $t 
\in \mathbb{R}$ yields the space $W^{*} = W\setminus \sim$ of trajectories modulo 
time-shift and the canonical projection $\pi:W \rightarrow W^{*}$. We write $W_U^* \subset W^*$ for the set of trajectories entering $U$. If $w^* \in W_U^*$, we write $s_U^+(w^*)$ for the trajectory in $W_+$ obtained by considering any $w\in W$ such that $\pi^*(w)=w^*$ and restricting $w$ to the trajectory in $W_+$ obtained after $w$ first enters $U$.

With these notations, we then introduce, for $f$ as in \eqref{eq:f1hatdef}, the function $F_f:W^{+} \rightarrow \mathbb{R}$ 
defined as
\begin{align}
F_f(w) = \int_0^{\infty} V_f(w(s))\mathrm{d}s, \quad V_f = - \frac{\Delta f}{f} \label{eq:FdefonW+}
\end{align}
where $\Delta f(x) = \frac{1}{2d} \sum_{|{e}|=1} f(x+e) - f(x)$ for $f :\Z^d \to \mathbb{R}$ 
denotes the discrete Laplacian. The integral \eqref{eq:FdefonW+} is well-defined because 
$V_f=0$ outside $U$, a finite set and the trajectory $w$ is transient. In turn, the function 
$F_f$ in \eqref{eq:FdefonW+} induces a function ${F}_f^{*}$ on  $W^{*} \times (0, \infty)$ 
given by
\begin{align}\label{eq:F^*}
{F}_f^{*}(w^{*},v) &= \begin{cases} F_f(s_{U}^+(w^{*})) &  \text{whenever } (w^{*},v) \in W^{*}_{U} \times [0,u]) \\
0 & \text{otherwise.}\end{cases}
\end{align}
with $U$ as in \eqref{eq:T_i}. The function  ${F}_f^{*}$ defines the exponential tilt of the 
interlacement measure $\P$, as follows. The measure $\P$ is formally defined on the space 
$\Omega = \{ \omega = \sum_i \delta_{(w_i^*,u_i)}: w_i^* \in W^*, \, u_i \in (0,\infty) \text{ for 
all } i\geq 0, \text{ and } \omega (W_K^* \times [0,u])< \infty \text{ for all } K \subset \subset 
\Z^d \text{ and } u>0\}$. For suitable $G: W^{*} \times (0, \infty)$ we write $\langle \omega 
,G \rangle $ for the canonical pairing, i.e.~the integral of $G$ with respect to the point 
measure $\omega \in \Omega$. In particular, one readily checks from \eqref{eq:F^*} that 
$\langle \omega ,F_f^* \rangle$ leads to a finite sum and hence is well-defined. We then 
define the measure $\widetilde{\P}_f$ via the Radon-Nikodym derivative
\begin{align}
\frac{d\widetilde{\mathbb{P}}_{f}}{d\mathbb{P}} = e^{\langle \omega ,F_f^* \rangle} \label{eq:PtildeAllLvlsdefn}.
\end{align}
We now gather a few essential features of $\widetilde{\mathbb{P}}_{f}$, which will be used 
in the sequel. By a slight adaptation of \cite[Proposition 2.1]{zbMATH06257634}, one has 
that $\widetilde{\mathbb{P}}_{f}$ defined by \eqref{eq:PtildeAllLvlsdefn} is a probability measure for any $f$ as in \eqref{eq:f1hatdef}. Moreover, denoting by $\nu \otimes 
\lambda$ the intensity measure on $W^* \times (0,\infty)$ of the interlacement process 
$\omega$ under $\P$, where $\lambda$ denotes Lebesgue measure, one has that the 
canonical point measure $\omega$ under $\widetilde{\mathbb{P}}_{f}$ is a Poisson point 
process with intensity measure $e^{F_f^*}(\nu \otimes \lambda)$. 

The measure $\widetilde{\mathbb{P}}_{f}$ retains an interlacement-like character. In particular, we will use the following fact in the sequel. For $\omega =\sum_i \delta_{(w_i^*, u_i)} \in \Omega$ and $K\subset \subset \Z^d$, let
\begin{equation}\label{def:muK}
\mu_K({\omega}) = \sum_{i\geq0} \delta_{(s_K^+(w_i^*), u_i)} {1}\{w_i^{*} \in {W}_K^{*}\},	
\end{equation}
a locally finite (by definition of $\Omega$) point measure on $W_+ \times \R_+$. In words, $\mu_K$ retains all labeled trajectories $(w_i^*, u_i)$ in the support of $\omega$ which enter $K$, and replaces $w_i^*$ for such points by the forward trajectory $s_K^+(w_i^*) \in W_+$ obtained after $w_i^*$ first enters $K$. Then (see \cite[(2.9)]{zbMATH06257634} for a similar result),
\begin{equation}
\label{eq:mu_K}\text{under $\widetilde{\mathbb{P}}_{f}$, $\mu_K({\omega}) $ is a PPP on $W^{+} \times \R_+$ with intensity measure $\widetilde{P}_{\tilde{e}_K}^f \otimes  \lambda $}
\end{equation}
 where $\widetilde{P}_{\tilde{e}_K}^f= \sum_{x} \tilde{e}_K(x) \widetilde{P}_{x}^f$, which we proceed to define. The measure $\widetilde{P}_{x}^f$ is given by
 \begin{align}
\frac{d\widetilde{P}_x^f}{dP_x} =  \frac{1}{f(x)}e^{\int_0^{\infty} V_f(X_s)ds}
\label{eq:tiltedPathDef}\end{align}
with $V_f$ as in \eqref{eq:FdefonW+}. On account of \cite[Lemma 1.2 and Corollary 1.3]{zbMATH06257634}, $\widetilde{P}_x^f$ is a probability measure for all $f$ as in 
\eqref{eq:f1hatdef}, and the canonical process $(X_t)_{t \geq0}$ under $\widetilde{P}_x^f$ is a Markov chain on $\Z^d$ with reversible measure 
$\tilde{\lambda}(x)=f^2(x)$, $x \in \Z^d$, whose semi-group on $L^2(\tilde{\lambda})$ is given by $(e^{t\tilde{\Delta}})_{t \geq 0}$ where $\tilde{\Delta}h(x)  = 
\frac{1}{2d} \sum_{|{e}|=1}\frac{ f(x+e)}{f(x)}(h(x+e)-h(x))$, for $h \in L^2(\tilde{\lambda})$. The equilibirum measure $\tilde{e}_K$ appearing in \eqref{eq:mu_K} is 
defined as
\begin{equation} \label{eq:eq-tilted}
 \tilde{e}_{K}(x)= \widetilde{P}_x^f[\tilde{H}_K = \infty] {1}_K(x)f(x)\Big(\frac{1}{2d} \sum_{|e|=1} f(x+e)\Big), \text{ for } x \in \mathbb{Z}^d.
\end{equation}
For later reference, we record that, upon introducing the tilted Green's function 
\begin{equation}\label{eq:g-tilted}\tilde{g}(x,y) = \frac{1}{f(y)^2} \widetilde{E}_x^f\Big[\int_0^{\infty} {1}\{ X_s = y\} \mathrm{d}s\Big],
\end{equation} 
one has in analogy with \eqref{eq:last-exit} that
\begin{equation}
\label{eq:greenseqidentity}
1= \sum_{y \in K } \tilde{g}(x,y)\tilde{e}_{K}(y), \text{ for all } x \in K.
\end{equation}

 Following are the two key properties of the measure $\widetilde{\mathbb{P}}_{f}$ defined by \eqref{eq:PtildeAllLvlsdefn}. 
 The first is a bound on the cost of tilting by $f$ in terms of the relative entropy $H(\widetilde{\mathbb{P}}_{f} \vert \P)\stackrel{\text{def.}}{=} \widetilde{\mathbb{E}}_{f}\big[ \log \big( \frac{\mathrm{d} \widetilde{\mathbb{P}}_{f}}{\mathrm{d} \P} \big) \big]$.
\begin{prop} \label{P:entropy} There exists $\Cl{C:entropy} \in (0,\infty)$ such that for all $u \in (0, \infty)$, if $d=3$, the bound
 \begin{equation}
 \label{eq:entropy-1}
 {\limsup_{\delta \downarrow 0 }} \  {\limsup_{\varepsilon \downarrow 0 }} \ {\limsup_{|x| \to 
 \infty }} \  \frac{\log|x|}{|x|}H(\widetilde{\mathbb{P}}_{f} \vert \P) \leq \Cr{C:entropy} (\sqrt{u} - 
\sqrt{u_*})^2  + C\eta^2
 \end{equation}
 holds for $f=f^{u(1-\eta \sqrt{\delta}),u; \varepsilon}$ and all $ \eta \in [0, c)$. Moreover, one 
 can choose $\Cr{C:entropy}=\frac\pi3$.
 \end{prop}

The second characteristic feature of the measure $\widetilde{\mathbb{P}}_{f}$  is the mass 
it assigns to the event $A^u(x,\delta)$ from \eqref{eq:A-final}, which involves the event 
$G^u$. In the next proposition all assumptions on $G^u$ made in Theorem~\ref{T:LB-gen} 
are tacitly assumed to hold; in particular, this includes \eqref{eq:LB-final-2}. Recall the 
dependence of $A^u(x,\delta)$ on the sprinkling parameter $\eta\in [0,1)$, which is implicit 
in our notation. 

\begin{prop} \label{P:typical} For all $\delta \in (0,1)$, $u \in (0, \infty)$, $ \eta \in [0, c)$ and $\varepsilon \in (0,u_* (1 \wedge \frac{ \eta \sqrt{\delta}}{10}))$, 
 \begin{align}
 &\label{eq:typical-1}  \lim_{|x| \to \infty} \widetilde{\mathbb{P}}_{f}[A^u(x,\delta)] =1 \  \text{ 
 	holds with $f=f^{u(1-\eta \sqrt{\delta}),u; \varepsilon}$.}
 \end{align}
 \end{prop} 
 
 The proofs of Propositions~\ref{P:entropy} and~\ref{P:typical} are  postponed to later sections. We now conclude the proof of Theorem~\ref{T:LB-gen} using these two results.
 
 \begin{proof}[Proof of Theorem~\ref{T:LB-gen}]
Let $u \in (0, \infty)$ and $f = f^{u(1-\eta \sqrt{\delta}),u; \varepsilon}$ for $\varepsilon > 0$. 
Since $\widetilde{\P}_{f}$ is 
 absolutely continuous with respect to $\P$ in view of \eqref{eq:PtildeAllLvlsdefn},~one 
 classically obtains by Jensen's inequality (see for instance the discussion following (2.7) in 
 \cite{BDZ95} for a proof) that for all $x \in \mathbb{Z}^d$ and $\delta \in (0,1)$,
 \begin{equation}\label{eq:gen.entrop.lower}
	\log	\P[A^u(x, \delta)]\geq \log (\widetilde{\P}_f[A^u(x, \delta)]) -\frac{H(\widetilde{\P}_f|\P)+e^{-1}}{\widetilde{\P}_f[A^u(x, \delta)]}.
\end{equation}
Multiplying by $\frac{\log|x|}{|x|}$ on both sides of \eqref{eq:gen.entrop.lower} and 
subsequently letting first $|x| \to \infty$, $\varepsilon \downarrow 0$ and $\delta \downarrow 
0$, the claim \eqref{eq:LB-final-3} (with $0 \leq \eta < c$) follows upon inserting the bounds 
\eqref{eq:entropy-1} and \eqref{eq:typical-1} on the right-hand side of 
\eqref{eq:gen.entrop.lower}.
 \end{proof}

 \section{
 	Relative entropy estimate}
 \label{sec:entropy}
 In the previous section, the proof of Theorem~\ref{T:LB-gen} was completed, subject to the validity of two results, stated as Propositions~\ref{P:entropy} and~\ref{P:typical}. In the present section we prove Proposition~\ref{P:entropy}. As will turn out, the estimate on the relative entropy will involve bounding potential theoretic quantities related to the tubes $T^i$, which recall are oriented towards $\frac{x}{|x|}$. Importantly, 
 the bounds derived need to be sufficiently sharp to make the correct (rotationally invariant!) 
 functional $\frac{\log |x|}{|x|}$ in \eqref{eq:entropy-1} appear in the large-scale limit. This is a 
 somewhat delicate matter when $x$ is in generic position and one cannot exploit (lattice) 
 symmetries.

 We start with some preparation.
 For $f: \Z^d \to \R$, we consider the Dirichlet form $\mathcal{E}(f,f)$ associated to the random walk, defined as
  \begin{equation}
  \label{eq:Dirichlet}
  \mathcal{E}(f,f)= \frac12 \sum_{ |x-y|=1} \frac1{2d} (f(x)-f(y))^2,
  \end{equation}
  where the sum ranges over all points $x,y \in \mathbb{Z}^d$ satisfying the constraint. Let 
  $H^2= \{ f: \Z^d \to \R:   \mathcal{E}(f,f) < \infty \}$. For $f,g \in H^2$, the energy $ 
  \mathcal{E}(f,g)$ is declared by polarization. Recall the functions $f =  {f}^{v,u; 
  \varepsilon}$ from \eqref{eq:f1hatdef}, which depend implicitly on both $x 
\in \mathbb{Z}^d$ and $\delta \in (0,1)$ via the choice of sets $T^i$ and $U$, 
cf.~\eqref{eq:T_i} and \eqref{eq:T_i-U}, entering their definition. Recall the notation for the (relative) capacity from below \eqref{eq:e_K,U}.
  
  \begin{lemma}   \label{L:Dirichlet} For all $\varepsilon > 0$, $0 < v \leq u$ and $f$ as in \eqref{eq:f1hatdef}, one has $f \in H^2$. Moreover, for $|x| \geq C(\delta)$,
 \begin{align}
 & \mathcal{E}(f, f) = v^{-1}\big[ \big(\sqrt{u_{*}+\varepsilon}-\sqrt{v}\big)^2 \textnormal{cap}_U(T^6) + \big(\sqrt{u_{*}+\varepsilon}-\sqrt{ v/u} \sqrt{u_{*}-\varepsilon}\big)^2 \textnormal{cap}_{T^3}(T^2) \big].  \label{eq:Dirichlet-2}
 \end{align}
  \end{lemma}
  \begin{proof}
 On account of \eqref{eq:chosenf-2} and \eqref{eq:T_i-U}, the function $f$ is constant 
 outside the finite set $U$ hence the sum in \eqref{eq:Dirichlet} is  effectively finite, whence 
 $f \in H^2$. We now show \eqref{eq:Dirichlet-2}. 
  
Introducing the shorthands $\alpha= ({\frac{u_{*}+\varepsilon}{v}})^{1/2} -1 $ and $\beta= 
({\frac{u_{*}+\varepsilon}{v})^{1/2} - (\frac{u_{*}-\varepsilon}{u}})^{1/2}$, the formula 
\eqref{eq:f1hatdef} defining $f$ reads $f - 1=  \alpha h_2- \beta h_1$. The functions 
$h_1(x) = P_x[H_{T^2} < T_{T^3}]$ and $h_2(x) = P_x[H_{T^6} < T_{U}]$ have the property 
that for any neighboring pair of points $x,y \in\Z^d$, $h_1(x)=h_1(y)$ or $h_2(x)=h_2(y)$ 
whenever $|x| \geq C(\delta)$; indeed for such $x$ we may assume in view of \eqref{eq:T_i} 
that the $1$-neighborhood of $T^3$ is contained in $T^6$. It follows that 
$h_2(x)=h_2(y)=1$ whenever at least one of the neighbors $x,y$ lies in $T^3$, whereas 
$h_1(x)=h_1(y)=0$ whenever $x,y \notin T^3$. All in all, it follows that 
$\mathcal{E}(h_1,h_2)=0$. Hence,
\begin{equation}
\label{eq:Dirichlet-3}
\mathcal{E}(f,f)= \mathcal{E}(f-1,f-1) = \alpha^2 \mathcal{E}(h_2,h_2) + \beta^2 \mathcal{E}(h_1,h_1).
\end{equation}
The claim \eqref{eq:Dirichlet-2} now follows from \eqref{eq:Dirichlet-3} using the classical fact (see for instance \cite[Prop.~1.9]{MR2932978} for a similar statement in a slightly different setup) that $\textnormal{cap}_{U}(K) = \mathcal{E}(V,V)$ with $V(x) = P_x[H_{K} < T_{U}]$ for all $U \subset \subset \mathbb{Z}^d$, $ K \subset U$, $K \neq \emptyset$.
  \end{proof}
  
Next we collect bounds for the relative capacities involved in Lemma~\ref{L:Dirichlet}. 
The proof of \eqref{eq:cap-3} below crucially involves 
Proposition~\ref{L:hit}.

 \begin{lemma} \label{L:cap}For all $\delta \in (0,\frac12)$, $|x| \geq C(\delta)$ and $i=1,\dots,6$, when $d=3$,
  \begin{align}
  &\label{eq:cap-1}  \textnormal{cap}(T^i) \leq (1+C \delta) \frac{\pi |x|}{3 \log|x|},\\
  &\textnormal{cap}_{U}(T^i) \leq \Big(1 + \frac{C}{\log |x|}\Big) \textnormal{cap}(T^i),   \label{eq:cap-2}\\[0.3em]
  &\textnormal{cap}_{T^3}(T^2)\leq C\delta^{-1} \textnormal{cap}(T^2).  \label{eq:cap-3}
  \end{align}
 \end{lemma} 
 \begin{proof}
 The bound \eqref{eq:cap-1} can be obtained by combining \cite[(2.24)]{prevost2023passage} and \cite[Lemma 2.2]{gosrodsev2021radius} (see, e.g., \cite[below (5.9)]{drewitz2023arm} for a similar argument). Next, we aim to show \eqref{eq:cap-2}. To do this, consider $y \notin U$ and notice that $\min_{z\in\partial T_i} |{y-z}| \geq 9|x|$ on account of \eqref{eq:T_i-U} whence $g(y,z) \leq C|x|^{-1}$ for $z \in T^i$. Using this, a last-exit decomposition, and \eqref{eq:cap-1}, it follows that $P_y[H_{T^i} < \infty] \leq C |x|^{-1} \text{cap}(T^i) \leq C' (\log |x|)^{-1}$. For $z \in T^i$, decomposing the (bare) equilibrium measure of $T^i$ over the exit location of $U$ and feeding this bound yields that
\begin{equation*}
 P_x[\widetilde{H}_{T^i} = \infty] = \sum_{y \in \partial U} P_x[\widetilde{H}_{T^i} > T_{U}, X_{T_{U}} = y] P_y[H_{T^i} = \infty] \\
\geq \Big(1 - \frac{C}{\log |x|}\Big) P_x[\widetilde{H}_{T^i} > T_{U}].
\end{equation*}
Summing both sides over $x\in  T^i$ and rearranging yields \eqref{eq:cap-2}. 
It remains to prove \eqref{eq:cap-3}. Using the inequality
\begin{equation}
\label{eq:cap-10}
e_{T^2}(z) \geq e_{T^2, T^3}(z) \inf_{y \notin T^3} P_y[H_{T^2}=\infty]
\end{equation}
valid for all $z \in \text{supp}(e_{T^2})$, which follows readily from \eqref{eq:e_K,U} upon applying the strong Markov property at time $T_{T^3}$, one obtains \eqref{eq:cap-3} by summing \eqref{eq:cap-10} over $z$ and using 
Proposition~\ref{L:hit}.
 \end{proof}
 
 Equipped with Lemmas~\ref{L:Dirichlet} and \ref{L:cap}, we are ready to proceed with the:
 
 \begin{proof}[Proof of Proposition~\ref{P:entropy}]
Let $f = f^{u(1 - \eta\sqrt{\delta}), u; \varepsilon}$. The relative entropy 
$H(\widetilde{\mathbb{P}}_{f} \vert \P)$ can be recast as
\begin{equation}\label{eq:entropy-pf-1}
H(\widetilde{\mathbb{P}}_{f} \vert \P)=\widetilde{\mathbb{E}}_{f}\Big[ \log \Big( \frac{\mathrm{d} \widetilde{\mathbb{P}}_{f}}{\mathrm{d} \P} \Big) \Big] 
\stackrel{\eqref{eq:PtildeAllLvlsdefn}}{=} \widetilde{\mathbb{E}}_{f}\big[ \langle \omega ,F_f^* \rangle \big] \stackrel{\eqref{eq:mu_K}}{=} u \widetilde{E}_{\tilde{e}_U}^f 
[F_f(X)],
\end{equation} 
where in applying \eqref{eq:mu_K} one notes that the Poisson variable $\langle \omega ,F_f^* \rangle$ depends on $\omega$ `through' $\mu_U$ (recall 
\eqref{def:muK}) in view of \eqref{eq:FdefonW+} and since $V_f$ vanishes outside $U$ and further only the pairs $(w^\ast, v)$ with $v \in [0, u]$ contribute to $\langle 
\omega ,F_f^* \rangle$ because of \eqref{eq:F^*} . By definition of $F_f$ in \eqref{eq:FdefonW+} one obtains, noting that all sums below are 
effectively finite, that
\begin{multline}\label{eq:entropy-pf-2}
\widetilde{E}_{\tilde{e}_U}^f [F_f(X)]= \int_0^{\infty} \widetilde{E}_{\tilde{e}_U}^f [V_f(X_s)] \mathrm{d}s \stackrel{\eqref{eq:g-tilted}}{=}\sum_{y,z} \tilde{e}_U(y)\tilde{g}(y,z) f^2(z) V_f(z)\\ 
\stackrel{\eqref{eq:greenseqidentity}}{=}\sum_{z }  f^2(z) V_f(z) \stackrel{\eqref{eq:FdefonW+}}{=} \sum_{z }  f(z) (-\Delta f)(z) \stackrel{\eqref{eq:Dirichlet}}{=} \mathcal{E}(f,f),
\end{multline}
where the last equality follows by a discrete version of the Gauss-Green theorem, see for instance \cite[Theorem 1.24]{MR3616731}. Combining \eqref{eq:entropy-pf-1}, \eqref{eq:entropy-pf-2} and \eqref{eq:Dirichlet-2} with the choice $v= u(1-\eta \sqrt{\delta})$ for $\eta \in [0,1]$, and subsequently feeding the bounds \eqref{eq:cap-2}-\eqref{eq:cap-3} in combination with \eqref{eq:cap-1}, one finds that 
\begin{multline*}
H(\widetilde{\mathbb{P}}_{f} \vert \P) \leq \frac{1+C \delta}{1-\eta \sqrt{\delta}} \, \frac{\Cr{C:entropy} |x|}{ \log|x|}\, \bigg[ \Big(\sqrt{u_{*}+\varepsilon}-\sqrt{u(1-\eta \sqrt{\delta})}\Big)^2 \Big(1 + \frac{C}{\log |x|}\Big)  \\
+C \delta^{-1}\Big(\sqrt{u_{*}+\varepsilon}-\sqrt{(1-\eta \sqrt{\delta})(u_{*}-\varepsilon)}\Big)^2 \bigg],
\end{multline*}
for all $u, \varepsilon > 0$, $\eta \in [0,c)$ and $|x| \geq C(\delta)$. Multiplying by $\frac{\log|x|}{|x|}$ on both sides, the desired bound \eqref{eq:entropy-1} follows upon taking the successive limits $|x| \to \infty$, $\varepsilon \downarrow 0$ and $\delta \downarrow 0$, using for the term in the second line that $|1- \sqrt{1-s}| \leq C s$ for $0<s< c$ with $s=\eta \sqrt{\delta}$, which upon squaring precludes the explosion of the factor $\delta^{-1}$ in the limit $\delta \downarrow 0$.
\end{proof}

  \section{Coupling tilted interlacements}
    \label{sec:coupling}

  The remainder of this article deals with the proof of Proposition~\ref{P:typical}, which concerns the effect of the tilted measure $\widetilde{\mathbb{P}}_{f}$ introduced in Section~\ref{sec:tilt}. Towards this goal, we first show that the tilted interlacements can be controlled locally in terms of regular interlacements but with a modified (and spatially inhomogenous) intensity close to $u_*$. Herein enter the specifics in our choice of tilt $f$ from Section~\ref{sec:tilt}.
  
The control is stated in terms of a coupling between the corresponding occupation time 
fields, which appears in Proposition~\ref{P:coupling}. This is the main result of this section. 
From Proposition~\ref{P:coupling} we first deduce Proposition~\ref{P:typical}. The proof of 
Proposition~\ref{P:coupling} then occupies the remainder of Section~\ref{sec:coupling}, and 
relies on two key intermediate results, Lemmas~\ref{L:tilt-compa} and~\ref{L:N-conc}. The 
former compares certain entrance laws for tilted vs.~untilted walk, while the latter exhibits 
concentration of certain associated excursion counts. Both of these results hinge on fine 
properties of the tilted random walk measure $\widetilde{P}_x^f$. Their proofs are given 
separately in Section~\ref{sec:tilt-hm}.
  
  Recall that, for a realization $\omega =\sum_i \delta_{(w_i^*, u_i)} \in \Omega$ of the Poisson process (declared under either $\mathbb{P}$ or $\widetilde{\mathbb{P}}_{f}$, cf.~\eqref{eq:PtildeAllLvlsdefn}), one introduces the occupation time at level $u>0$ and $x \in \Z^d$ as 
 \begin{equation}
 \label{eq:occ-time}
 \ell_x^u= \ell_x^u(\omega)= \sum_{i : u_i \leq u} \int_{-\infty}^\infty 1\{ w_i^*(t)=x\} \, \textrm{d} t.
 \end{equation}
 To distinguish between the two possible reference measures, we will henceforth write 
 $\tilde{\ell}_{\cdot}^u$ to refer to a random field having the same law as ${\ell}_{\cdot}^u$ 
 under $\widetilde{\mathbb{P}}_{f}$. We consider boxes 
 \begin{equation}\label{eq:B_i}
 B_i^z= B\big(z,r_0^{({i+2})/{8}}\big), \text{ for $z \in \mathbb{Z}^d$ and $i\in \{1,\dots,4\}$,}
 \end{equation} 
 where $r_1$ is given by \eqref{eq:LB-final-1}. We will almost exclusively consider centers $z \in \Gamma=\Gamma_{\text{int}} \cup \Gamma_{\text{ext}} $, where 
 $\Gamma_{\text{int}}= T(x)$ and $\Gamma_{\text{ext}}  = \partial T^4$, with $T(x)$ and $T^4$ as in \eqref{eq:T_x} and \eqref{eq:T_i}. Note that, with the choices of 
 radii in \eqref{eq:B_i}, whenever $|x| \ge C(\delta) $ we have that $B_i^z \subset B_4^z \subset T^1$ for all $z \in \Gamma_{\text{int}}$, and $B_i^z\subset (T^5 
 \setminus T^3)$ for all $z \in \Gamma_{\text{ext}}$. The following proposition essentially asserts that the effect of tilting is to make the law of the (tilted) occupation 
 times at level $u$ look  slightly super-critical in the region $T^1$, i.e.~roughly like untilted interlacement occupation times at level $u_*- O(\varepsilon)$, and slightly 
 sub-critical in the region $T^5 \setminus T^3$. 


 \begin{prop}
 \label{P:coupling} 
  For all $\delta \in (0,1)$, 
  $\varepsilon \in (0,\frac{u_*}{10} \wedge 1)$ and $ \eta \in [0, \Cl[c]{c:eta-coup})$, 
 the following holds.
  \begin{itemize}
 \item[i)]  If $z \in \Gamma_{\textnormal{ext}} $, then with $f = f^{v,u; \varepsilon}$, 
 $v=u(1-\eta \sqrt{\delta})$ and abbreviating $B=B_1^z$, there exists a 
 coupling $Q_B$ of $(\ell^{u_*+\frac{1}{2}\varepsilon}_x)_{x\in B}$, $(\ell^{u_*+\frac{3}{2}\varepsilon}_x)_{x\in B}$ (each under $\P$) and $(\tilde\ell^{v}_x)_{x\in B}$ 
 (having the same law as $(\ell^v_x)_{x\in B}$ under $\tilde{\P}_f$ by above convention), 
  such that for $\tilde c=\tilde c(u,\delta,\varepsilon,\eta)$,
 \begin{equation}
 \label{eq:coupling-0}
 Q_B\big[  \ell^{u_*+\frac{1}{2}\varepsilon}_x \leq \tilde\ell^v_x \leq \ell^{u_*+\frac{3}{2}\varepsilon}_x , \, x \in B \big] \geq 1 -e^{-\tilde cr_0^{\tilde c}}.
 \end{equation}

  \item[ii)]  If $z \in \Gamma_{\text{int}} $, with $s_i(t)= t(1- \gamma_i)^{\sigma_i}$ for 
  $\gamma_i \in [0,1)$, $\sigma_i \in \{\pm1\}$, $i=1,2$, then one has for $f$ as above that
there exists a coupling $Q_B$ of 
 $(\ell^{s_i(u_*-\frac{2k-1}{2}\varepsilon)}_x)_{x\in B,\, k,i \in\{1,2\}}$ and $(\tilde\ell^{s_i(u)}_x)_{x\in B, i=1,2}$, such that, for suitable $\tilde c=\tilde c(u,\delta,\varepsilon,\eta,\gamma_i,\sigma_i)$,
 \begin{equation}
 \label{eq:coupling-1}
 Q_B\big[
  \ell^{s_i(u_*-\frac{3}{2}\varepsilon)}_x \leq \tilde\ell^{s_i(u)}_x \leq \ell^{s_i(u_*-\frac{1}{2}\varepsilon)}_x  , \, x \in B, \, i=1,2\big] \geq 1 -e^{-\tilde{c} r_0^{\tilde{c}}}, 
 \end{equation}
 \end{itemize} 
 \end{prop}

  With the aid of Proposition~\ref{P:coupling}, we first give the proof of Proposition~\ref{P:typical}.
  
  \begin{proof}[Proof of Proposition~\ref{P:typical}]
  We will freely (and tacitly) assume that various statements hold for $|x| \ge C(\delta)$, which is no loss of generality in view of \eqref{eq:typical-1} since the latter concerns the 
  limit $|x|\to \infty$ only. Let $f = f^{v,u; \varepsilon}$ with $v=u(1-\eta 
  \sqrt{\delta})$ and $\eta \in [0,\Cr{c:eta-coup})$ as for the conclusions of Proposition~\ref{P:coupling} to hold. The following
 conclusions will always hold uniformly in $\eta$ as above, $\delta \in (0,1)$ and $\varepsilon \in (0,u_* (1\wedge \frac{ \eta \sqrt{\delta}}{10}))$, as postulated in the statement of Proposition~\ref{P:typical}. Using Proposition~\ref{P:coupling} we argue separately that
  \begin{align}
& \lim_{|x| \to \infty} \widetilde{\P}_f[\lr{}{\mathcal V^{v}}{T^4}{\partial T^5}]=0, \label{eq:typical-20}\\
 & \lim_{|x| \to \infty} \widetilde{\P}_f[G^u, \lr{}{\mathcal V^u \cap T^1}{B(0,r_1)}{B(x,r_1)}]=1.\label{eq:typical-21}
  \end{align}
  Recalling $A^{u}(x, \delta)$ from \eqref{eq:A-final}, the claim \eqref{eq:typical-1} 
  immediately follows from \eqref{eq:typical-20} and \eqref{eq:typical-21}. To obtain 
  \eqref{eq:typical-20} applying the relevant coupling from 
  Proposition~\ref{P:coupling},\textit{i)}, one finds that for all $z \in \Gamma_{\text{int}} $, 
  with $B=B_1^z$,
  \begin{multline*}
  \widetilde{\P}_f[\lr{}{\mathcal V^{v}}{z}{\partial B} ]=Q_B [\lr{}{\{\tilde{\ell}_\cdot^v =0 \}}{z}{\partial B} ] \stackrel{\eqref{eq:coupling-0}}{\leq} Q_B [\lr{}{\{\ell^{u_*+\frac{1}{2}\varepsilon}_{\cdot} =0 \}}{z}{\partial B} ] + e^{-c|x|^{c'(\delta)}} \\=  \P\big[\lr{}{\mathcal V^{u_*+\frac{1}{2}\varepsilon}}{z}{\partial B} \big] + e^{-cr_0^{\tilde{c}}} \leq Ce^{-c|x|^{c'(\delta)}},
  \end{multline*}
  where the last step follows using \cite[Theorem~1.2,i)]{RI-I} and recalling that $r_0 \geq c|x|^{\delta}$. Applying a union bound over $z \in \partial T^4 (=\Gamma_{\text{ext}})$ and using the bound on $ \widetilde{\P}_f[\lr{}{\mathcal V^{v}}{z}{\partial B} ]$ just obtained yields \eqref{eq:typical-20}. To deal with \eqref{eq:typical-21}, we use part \textit{ii)} of Proposition~\ref{P:coupling}. We consider the two events appearing in \eqref{eq:typical-21}
 separately, and start with the connection event appearing there. Throughout the rest of the proof, constants may implicitly depend on all of $u,\delta,\varepsilon,\eta$.
 
 Let $A_z= B(z,r_1)\setminus B(z,r_1/2)$ and $\text{LU}_z(\mathcal{V}, \mathcal{V}')$ denote the 
 event that for any $z'$ with $|z-z'| \leq r_1/4$, any two connected crossing components of both $\mathcal{V} \cap A_z$ and $\mathcal{V} \cap A_{z'}$ are connected in 
 $\mathcal{V}' \cap (B(z,2r_1) \setminus B(z,r_1/2))$; cf.~\cite[(3.11)-(3.12)]{MR3602841} for similar events. One checks that $\text{LU}_z(\mathcal{V}, \mathcal{V}')$ is decreasing in $\mathcal{V}$ and increasing in $\mathcal{V}'$. For any $v \geq u$, the joint occurrence as $z$ varies over $\Gamma_{\text{int}}=T(x)$ of i) the event $\text{LU}_z(\mathcal{V}^v, 
 \mathcal{V}^u)$, and ii) the event that 
 $A_z \cap \mathcal{V}^v$ has at least one crossing component, is seen to imply $E$. Let $v 
 = \frac{u}{1- \gamma}$, with $\gamma$ chosen small enough so that $w < u_* - \frac14 
 \varepsilon$, where $w= \frac{ u_*-\frac{1}{2}\varepsilon}{1- \gamma}$. Applying 
 Proposition~\ref{P:coupling},\textit{ii)} with $\gamma_1 =0$ and $\gamma_2=\gamma$, 
 $\sigma_2=-1$, one finds that for any $z \in \Gamma_{\text{int}}$,  \begin{equation}
 \widetilde{\P}_f[\text{LU}_z(\mathcal{V}^v, \mathcal{V}^u)] \geq \P[\text{LU}_z(\mathcal{V}^{w}, \mathcal{V}^{u_*-\frac32\varepsilon})]  -e^{-\tilde{c} r_0^{\tilde{c}}}. \label{eq:typical-22}
 \end{equation}
  But as a straightforward consequence of \cite[Theorem~1.2,ii)]{RI-I}, which applies because $w< u_*$, one obtains that the right-hand side of \eqref{eq:typical-22} exceeds $1- Ce^{-\tilde{c} r_0^{\tilde{c}}}$. One deals with the events in ii) above, which are increasing in $\mathcal{V}^v$, in a similar (but simpler) fashion. All in all, together with a union bound, this implies \eqref{eq:typical-21} in absence of the event $G^u$, i.e.~$\widetilde{\P}_f[E] \to 1$ as $|x|\to \infty$, where $E= \{\lr{}{}{B(0,r_1)}{B(x,r_1)} \text{ in } \mathcal V^u \cap T^1\}$. To accommodate the presence of $G^u$ in \eqref{eq:typical-21}, recalling its form specified above Theorem~\ref{T:LB-gen} and applying a union bound, one sees that it is enough to argue that 
  \begin{equation}
   \lim_{|x| \to \infty} \,  |I| \cdot  \sup_{i}  \widetilde{\P}_f[(G_{B_i})^c]=0.
  \label{eq:typical-23}
 \end{equation}
  This is obtained by combining \eqref{eq:LB-final-2} and Proposition~\ref{P:coupling},\textit{ii)}. Indeed, recall to this effect (cf.~above Theorem~\ref{T:LB-gen}) that $B_i$
 is a box of radius $10r_1$ with center $z \in \Gamma_{\text{int}}$. By choice of $r_1$ in \eqref{eq:LB-final-1} and in view of \eqref{eq:B_i}, this means that $B_i \subset B_1^z=B$. In particular, the event $G_{B_i}$ is thus measurable relative to $(\tilde\ell^{s_i(u)}_x)_{x\in B, i=1,2}$, where $s_i(u)= u(1- \gamma_i)$ for suitable choice of $\gamma_i \in \{ \frac k4\eta \sqrt{\delta}: k=0,1,2,3,4\}$. Assume for concreteness that $s_1(u)> s_2(u)$, so that $G_{B_i}$ is increasing in $(\tilde\ell^{s_1(u)}_x)_{x\in B}$ and decreasing in $(\tilde\ell^{s_2(u)}_x)_{x\in B}$. With $(v,w)=(s_1(u), s_2(u))$, $(v',w')= (s_1(u_* -\frac32\varepsilon), s_2(u_* -\frac12\varepsilon))$, it then follows by application of \eqref{eq:coupling-1}, using the (separate) monotonicity of $G_{B_i}= G_{B_i}^{v,w}$, that 
 $$\widetilde{\P}_f[G_{B_i}^{v,w}]\geq \P[G_{B_i}^{v',w'}] -e^{-\tilde{c} r_0^{\tilde{c}}}.$$
 The claim \eqref{eq:typical-23} now follows using \eqref{eq:LB-final-2}, noting in particular that the assumption $\varepsilon < u_* \frac{ \eta \sqrt{\delta}}{10}$ guarantees that $v'>w'$.
  \end{proof}
  
It thus remains to prove Proposition~\ref{P:coupling}. We start with some preparation. For $z \in \Gamma$, with $B= B_1^z$ and $U= B_2^z$ (recall \eqref{eq:B_i}) we 
will consider both for tilted and untilted walks successive excursions between $B$ and $U^c$. In light of this, we introduce a reference Poisson point process $\eta_B = 
\sum_{n \geq 0} \delta_{(\zeta_n, u_n)}$ governed by the probability measure ${Q}_B$ on the state space $\Xi \times \mathbb{R}_{+}$, where $\Xi=\Xi_{B,U}$ denotes 
the space of relevant excursions, i.e.~of finite length nearest neighbour trajectories starting on $\partial B$, with range contained in $U$ up to their terminal point, a 
vertex in $U^c$. 
The intensity measure of $\eta$ is $\nu_{B} \otimes \lambda$, where $\lambda$ denotes the Lebesgue measure and
\begin{align}\nu_{B}(\cdot) \stackrel{\text{def.}}{=} \sum_{x \in \partial B} P_x\big[(X_t)_{0 \leq t \leq T_{U}} \in \cdot \big].
 \label{eq:muB1intro}
\end{align}
 Importantly for what is to follow, if $f=f^{v,u; \varepsilon}$ for any $0<  v \leq u$ and $\varepsilon \in (0,\frac{u_*}{10} \wedge 1)$, then regardless of the choice 
 of $z \in \Gamma$, by \eqref{eq:chosenf-2} the function $f$ is constant in the one-neighborhood of $U$, which in turn implies that $V_f(x)=0$ for all $x \in U$. 
 In view of \eqref{eq:tiltedPathDef}, this means that $P_x$ in \eqref{eq:muB1intro} can be freely replaced by $\widetilde{P}_x^f$, i.e.~excursions between $B$ and 
 $U^c$ do not witness the tilt.
  
We now proceed to define two Markov chains $Z=(Z_n)_{n \ge 1}$ and $\widetilde{Z}=(\widetilde{Z}_n)_{n \geq1}$ on $\Xi$, as follows. To this effect, we introduce the entrance distribution and potential of $B$, for $x,y \in \Z^d$, as
\begin{equation}\label{eq:h_B}
h_B(x,y)= P_x\big[H_B < \infty, X_{H_B}=y\big], \quad h_B(x)= \sum_y h_B(x,y). 
\end{equation}
The corresponding tilted quantities $\tilde{h}_B(x,y)$ and $\tilde{h}_B(x)$ are obtained by replacing $P_x$ by the tilted measure $\widetilde{P}_x \equiv \widetilde{P}_x^f$. Both $Z$ and $\widetilde{Z}$ are specified in terms of their transition densities $\pi$ and $\tilde{\pi}$ relative to $\nu_{B}$ in \eqref{eq:muB1intro},~that is,
\begin{equation}\label{eq:exc1}
{Z}_1 \stackrel{\text{law}}{=} {\pi}_0(\zeta)\nu_{B}(d\zeta), \quad P[{Z}_{k+1} \in d\zeta | {Z}_1,\dots, {Z}_k ] = \pi({Z}_k, \zeta) \nu_{B}(d\zeta), \text{ for } k \geq 1
 \end{equation}
  and similarly for $\widetilde{Z}$ with $\tilde{\pi}$ in place of ${\pi}$, where, for $\zeta=(\zeta_0,\dots,\zeta_n) \in \Xi$ and with $\bar{e}_B=e_B/\text{cap}(B)$ the normalized equilibrium measure,
  \begin{equation}\label{eq:exc2}
  \pi_0(\zeta) \stackrel{\text{def.}}{=} \bar{e}_{B}(\zeta_0), \quad \pi (\zeta,\zeta')\stackrel{\text{def.}}{=} h_{B}(\zeta_n, \zeta_0') + \bar{e}_{B}(\zeta_0') (1-h_B)(\zeta_n).
  \end{equation}
  The corresponding tilted transition densities are declared by replacing $\bar{e}_{B}$ and $h_{B}$ by their tilted analogues. One readily sees from \eqref{eq:exc1} and \eqref{eq:exc2} in combination with the observation made below \eqref{eq:muB1intro} that $Z$, resp.~$\widetilde{Z}$ has the same law as the excursions from $B$ to $U^c$ induced by the interlacement process $\omega$ under $\P$,~resp.~${\widetilde{\P}_f}$, when ordering them according to increasing label $u$ and in order of appearance within a given random walk trajectory in the support of the point measure $\omega$. The key behind the coupling(s) postulated in Proposition~\ref{P:coupling} is encapsulated in the following comparison of transition densities. 
  
Before we state it concisely, observe that, in light of the statement of Proposition~\ref{P:coupling}, when working with $f \in \mathcal{F}$ we only ever have to deal with 
boxes $B=B_1^z$ having centers $z \in \Gamma$. Also note that, albeit implicit in our notation, the set of centers $\Gamma = \Gamma(x)$ depends on $x \in 
\mathbb{Z}^d$ via the oblique cylinders introduced in \eqref{eq:T_i}. These cylinders, as well as the boxes $B_i^z$, depend on one further parameter $\delta \in 
(0,\frac16)$.

  \begin{lemma} \label{L:tilt-compa} For all $f$ as in \eqref{eq:f1hatdef}, $x \in \mathbb{Z}^d$, $z \in \Gamma(x)$, with $\Xi=\Xi_{B_1^z,B_2^z}$,
  \begin{equation}
  \label{eq:tilt-compa1}
  \sup_{\zeta,\zeta' \in \Xi} \bigg\{ \bigg| \frac{\tilde{\pi}_0(\zeta)}{{\pi_0}(\zeta)} -1 \bigg| , \,   \bigg| \frac{\tilde{\pi}(\zeta',\zeta)}{{\pi}(\zeta',\zeta)} -1 \bigg| \bigg\} \leq Cr_0^{-c}. 
  \end{equation}
  \end{lemma}
  The proof of Lemma~\ref{L:tilt-compa} requires some work and is postponed to the next section. The main issue is that between successive excursions, the walk under $\widetilde{P}_x$, $x \in B$, may in principle travel far (with polynomial probability in $r_0$), thereby exploring regions in which the effect of the tilt is severely felt. An additional source of difficulty stems from the fact that \eqref{eq:tilt-compa1} requires a (stronger) control of ratios rather than of mere differences $|\tilde{\pi} -\pi|$.
 
 Assuming Lemma~\ref{L:tilt-compa} to hold, the method of soft local time \cite{PopTeix}, see also~\cite{comets2013large} which will be sufficient for our purposes, allows to define under $Q_B$ an event $U_n^{v}$ for each integer $n \geq 1$ and $v \in (0,c)$ such that, whenever $|x| \geq C(v)$ (so that the bound in \eqref{eq:tilt-compa1} is smaller than $\frac{v}{3}$), one has
 \begin{equation}
 \label{eq:exc3}
 \begin{split}
 &Q_B[U_n^{v}] \geq 1 - C \exp(-c vn), \text{ and }\\
 & \text{on } U_n^{v}, \text{ for all $m \geq n$:} \begin{array}{l}
\{Z_1, \ldots, Z_{(1-v)m} \} \subset \{\widetilde{Z}_1, \ldots, \widetilde{Z}_{(1+4v)m}\} \\
\{\widetilde{Z}_1, \ldots, \widetilde{Z}_{(1-v)m} \} \subset \{Z_1, \ldots, Z_{(1+4v)m}\}.
\end{array}
 \end{split}
 \end{equation}
 To obtain \eqref{eq:exc3} one retraces the steps of \cite[Lemma 2.1]{comets2013large}, and the key assumption~\cite[(2.5)]{comets2013large} appearing in that context is replaced by \eqref{eq:tilt-compa1}. 
 
 Next, we attach to each of $Z$ and $\widetilde{Z}$ a sequence $\sigma =(\sigma_{k})_{k \geq 1}$ and $\tilde{\sigma}$ of labels in $\{0,1\}$, as follows. We $\sigma_1= \tilde{\sigma}_1=1$ and for each $k \geq 2$, the label $\sigma_k$ has conditional law given $Z_{k-1}, Z_k$ given by 
 \begin{equation}\label{eq:sigma-k}
 P[\sigma_k=0 \,| \, Z_{k-1}, Z_k] = 1- P[\sigma_k=1 \,| \, Z_{k-1}, Z_k] = \frac{h_{B}(Z_{k-1}^e, Z_k^{i})}{\pi(Z_{k-1},Z_k)},
 \end{equation}
 where $Z_{\cdot}^{i/e}$ refer to the initial/end-point of $Z_{\cdot}$. The prescription for $\tilde{\sigma}_k$ is identical to \eqref{eq:sigma-k} but using $\tilde{h}_B$ and $\tilde{\pi}$ instead.  The label $\sigma$ recovers information about the trajectories underlying the excursions forming $Z$: the label $\sigma_k=1$ signals excursions $Z_k,\dots$ stemming from a new random walk trajectory. We henceforth assume that $Q_B$ is suitably enlarged as to carry the sequences $\sigma$ and $\tilde{\sigma}$ with the correct law, independently of each other conditionally on $Z,\widetilde{Z}$. 

As a last ingredient, we assume $Q_B$ to carry additionally two independent Poisson counting processes $n_B$ and $\tilde{n}_B$ on $[0,\infty)$ with intensity $\text{cap}(B)$ and $\widetilde{\text{cap}}(B)$, respectively, and write $n_B(t)= n_B([0,t])$ for $t \geq 0$, a Poisson variable with mean $\text{cap}(B)t$. We consider the random variables
\begin{equation}
\label{eq:N}
\begin{split}
&\mathcal{N}^u = \sup \big\{ n \geq1 : \textstyle \sum_{1\leq k \leq n} \sigma_k \leq n_B(u)\big\}, \quad u >0,
\end{split}
\end{equation}
(with the convention $\sup \emptyset =0$), and similarly $\widetilde{\mathcal{N}}^u$, using $\tilde\sigma_k$ and $\tilde n_B$ instead. These random variables admit the 
following comparison estimates, which depend on the function $f$ underlying the law of $\widetilde{Z}$ determining $\widetilde{\mathcal{N}}^u$. 


\begin{lemma}\label{L:N-conc} For all $\delta \in (0,1)$, $\varepsilon \in (0,\frac{u_*}{10} \wedge 1)$ and $ \eta \in [0, \Cl[c]{c:eta-coup})$, the following hold.
\begin{itemize}
\item[i)]  If $z \in \Gamma_{\textnormal{ext}} $, then with $f=f^{v,u; \varepsilon}$, $v=u(1-\eta \sqrt{\delta})$, 
 \begin{equation}
 \label{eq:N-conc1}
 Q_B\big[ (1+ c \varepsilon)  {\mathcal{N}}^{u_*+\frac{1}{2}\varepsilon} \leq \widetilde{\mathcal{N}}^v \leq (1+c\varepsilon)^{-1} {\mathcal{N}}^{u_*+\frac{3}{2}\varepsilon}  \big] \geq 1 -e^{-\tilde cr_0^{\tilde c}}.
 \end{equation}
  \item[ii)]  If $z \in \Gamma_{\text{int}} $, then with $s_i(t)= t(1- \gamma_i)^{\sigma_i}$ for $\gamma_i \in [0,1)$, $\sigma_i \in \{\pm1\}$, $i=1,2$, one has for $f$ as above that 
   \begin{equation}
 \label{eq:N-conc2}
 Q_B\big[ (1+ \varepsilon)
  {\mathcal{N}}^{s_i(u_*-\frac{3}{2}\varepsilon)} \leq \widetilde{\mathcal{N}}^{s_i(u)} \leq (1+ \varepsilon)^{-1} {\mathcal{N}}^{s_i(u_*-\frac{1}{2}\varepsilon)}  \big] \geq 1 -e^{-\tilde{c} r_0^{\tilde{c}}}, \  i=1,2,
 \end{equation}
 for suitable $\tilde c$ depending on $u,\delta,\varepsilon,\eta,\gamma_i,\sigma_i$.
 \end{itemize} 
\end{lemma}

The bounds \eqref{eq:N-conc1} and \eqref{eq:N-conc2} are tailored to our purposes. Underlying them are similar controls on the behavior of the tilted walk $\widetilde{P}_x^f$ that are also needed to prove \eqref{eq:tilt-compa1} (essentially because the relevant quantities $\tilde{\pi}$ and $\tilde{h}_B$ also crucially appear in \eqref{eq:sigma-k} and thus govern the law of the variables $(\tilde\sigma_k)_{k \geq1}$ entering $\widetilde{\mathcal{N}}^u$ in \eqref{eq:N}). The proof of Lemma~\ref{L:N-conc} thus appears jointly with that of Lemma~\ref{L:tilt-compa} in the next section. With both Lemmas~\ref{L:N-conc} and \ref{L:tilt-compa} at our disposal, we are ready to give the short:

\begin{proof}[Proof of Proposition~\ref{P:coupling}]
We choose $Q_B$ the coupling constructed above. Recall from \eqref{eq:mu_K} the induced interlacement process $\mu_B = \mu_B(\omega)$  (declared under either of $\P$ and $\widetilde{P}_f$), collecting the labeled trajectories entering $B$ after their entrance time in $B$. Observe that under $Q_B$,
the random measures 
\begin{equation}\label{eq:coup-41}
\big(\xi^u, \,\tilde\xi^v\big)_{u,v>0}\stackrel{\text{def.}}{=}
\Big( \sum_{1 \leq k \leq \mathcal{N}^u} \delta_{Z_k}, \, \sum_{1 \leq k \leq \widetilde{\mathcal{N}}^v} \delta_{\widetilde{Z}_k}\Big)_{u,v>0}
\end{equation}
have the same law as the excursions between $B$ and $U^c$ induced by the trajectories in the support of $\mu_B$ with label at most $u$ and $v$ under $\P$ and $\widetilde{\P}_f$, respectively.  As can be seen from \eqref{eq:occ-time}, the occupation time field $(\ell_x^u)_{x\in B, u>0}$ is clearly a measurable function of the excursions induced by $\mu_B(\omega)$ under $\P$, and similarly for $(\tilde \ell_x^u)_{x\in B, u>0}$ under $\widetilde{\P}_f$ (recall our notational convention from below \eqref{eq:occ-time}). Hence, $\ell_x^u, \tilde \ell_x^v$, $x\in B, u,v>0$ can be viewed (in law) as functionals of the point measures in \eqref{eq:coup-41}, as
\begin{equation}
\label{eq:coup-42}
\ell_{x}^u \stackrel{\text{law}}{=} \ell_x(\xi^u) \stackrel{\text{def.}}{=}\sum_{1 \leq k \leq \mathcal{N}^u} \int_0^{\text{len}(Z_k)} 1\{Z_k(t)=x\} \mathrm{d} t, \quad x \in B, \, u>0,
\end{equation}
where $\text{len}(Z_k)$ refers to the length (duration) of the excursion $Z_k$, and similarly for $\tilde\ell_{x}^{u}$. In particular, $Q_B$ thus furnishes a coupling of the occupation time fields $(\ell_{x}^{u}, \tilde\ell_{x}^{v}: x \in B, u,v >0)$. 

 We now argue that \eqref{eq:coupling-0} holds (which implicitly entails that $z$, $f$ and $B=B_1^z$ have been chosen accordingly). For two point measures $\xi,\xi'$ on the excursion space $\Xi$, as in \eqref{eq:coup-41} for instance, we write $\xi \leq \xi'$ if $\text{supp}(\xi) \subset \text{supp}(\xi')$. As follows plainly from \eqref{eq:coup-42}, the occupation time field is clearly monotone with respect to this order, i.e.~$\xi \leq \xi'$ implies $\ell(\xi) \leq \ell(\xi')$. Fix $v= c \varepsilon$ with $c$ small enough so that the conclusions of \eqref{eq:exc3} hold and $\frac{1+4v}{1-v} \leq 1+ c\varepsilon$. Let $n=  \mathcal{N}^{u_*}$. With these choices for $v$ and $n$, if the event $U_n^v$ appearing in \eqref{eq:exc3} and the event on the left-hand side of \eqref{eq:N-conc1} jointly occur (under $Q_B$), one deduces using $\mathcal{N}^{u_*+\frac{1}{2}\varepsilon} \geq \mathcal{N}^{u_*} $ when applying \eqref{eq:exc3} that
 $$
 \{Z_1, \ldots, Z_{{\mathcal{N}}^{u_*+\frac{1}{2}\varepsilon}}\} \stackrel{\eqref{eq:exc3}}{\subset} 
  \{\widetilde{Z}_1, \ldots, \widetilde{Z}_{(1+ c\varepsilon) {\mathcal{N}}^{u_*+\frac{1}{2}\varepsilon} }\} \stackrel{\eqref{eq:N-conc1}}{\subset} \{\widetilde{Z}_1, \ldots, \widetilde{Z}_{\widetilde{\mathcal{N}}^v}\}.
 $$
 With a view towards \eqref{eq:coup-41}, this yields that $\xi^{u_*+\frac{1}{2}\varepsilon} \leq 
 \tilde\xi^v$ and hence $\ell_x^{u_*+\frac{1}{2}\varepsilon}= 
 \ell_x(\xi^{u_*+\frac{1}{2}\varepsilon}) \leq {\ell}_x(\tilde\xi^v)=\tilde\ell_x^v$ for all $x \in B$ 
 by monotonicity. The other inequality $\tilde\ell^v_x \leq \ell^{u_*+\frac{3}{2}\varepsilon}_x$ 
 inherent to the event in \eqref{eq:coupling-0} is obtained similarly, now exploiting 
 \eqref{eq:exc3} together with the second inequality in \eqref{eq:N-conc1}. To conclude the 
 proof of \eqref{eq:coupling-0} it thus suffices to combine the bound on the event in 
 \eqref{eq:exc3}, together with a suitable estimate on $Q_B[U_{n= \mathcal{N}^{u_*} 
 }^{c\varepsilon}]$. The latter is obtained by combining \eqref{eq:exc3}, first for $|x| \geq 
C(\varepsilon)$ but eventually for all $x$ by possibly adapting the constant $\tilde{c}$, and 
the fact (see, for instance, below \eqref{eq:conc-11}) that $Q_B[ \mathcal{N}^{u_*} \geq 
\frac12u_* \text{cap}(B)] \geq 1 -e^{-c \cdot \text{cap}(B)}$.
 
 The proof of item \textit{ii)} in Proposition~\ref{P:coupling} follows a similar reasoning as that of item \textit{i)}, now combining \eqref{eq:exc3} and \eqref{eq:N-conc2} to deduce \eqref{eq:coupling-1}.
\end{proof}

\section{Tilted harmonic measure}
\label{sec:tilt-hm}
In this section we prove Lemmas~\ref{L:tilt-compa} and~\ref{L:N-conc}, which concern the 
tilted random walk measure $\widetilde{P}_x^f$ introduced in \eqref{eq:tiltedPathDef}. The 
function $f$ is defined in \eqref{eq:f1hatdef} and implicitly depends on parameters $u,v$ and $\varepsilon$, as well as on $x\in \mathbb{Z}^d$ and $\delta \in 
(0,\frac16)$, which determine the underlying (oblique) regions $T^i$, cf.~\eqref{eq:T_i} and 
\eqref{eq:f1hatdef}. In order to keep notation reasonable, whenever the parameters $\delta, 
u,v,\varepsilon$ are not further specified below, it is tacitly assumed that the conclusions 
hold for all $\delta \in (0,\frac16)$, $\varepsilon \in (0,\frac{u_*}{10} \wedge 1)$, and all $0 < 
v \le u$. Recall that, when working with $f$ as in \eqref{eq:f1hatdef}, we only have to deal with boxes $B = B_1^z$ having centers $z \in \Gamma = \Gamma(x)$ (see 
above Proposition~\ref{P:coupling} and Lemma~\ref{L:tilt-compa} for notation).

\smallskip
The following result will be key.
\begin{lemma} \label{lem:tilted_beta_bound}
For $f$ as in \eqref{eq:f1hatdef}, 
\begin{align}\label{eq:beta-0}
\beta^{f}(x) \stackrel{\textnormal{def.}}{=} \sup_{i=1,2} \sup_{z \in \Gamma(x), \, y \in \partial B_{i+1}^z} \widetilde{P}_y^f[H_{B_i^z} < \infty] \to 0 \text{ as } |x| \rightarrow \infty.
\end{align}
\end{lemma} 
\begin{proof} 
For simplicity we assume that $i=2$ in the sequel. The other case is treated in the same way. 
Let $f$ be as in \eqref{eq:f1hatdef}. Recall from \eqref{eq:FdefonW+} that $V_f= -\Delta f/f$. We will prove with $y$ ranging over 
$\bigcup_{z\in \Gamma} \partial^{\text{out}} B_4^z$ below (noting that this range depends 
on $x \in \Z^d$ as $\Gamma = \Gamma(x)$) that
\begin{equation}\label{eq:beta-1}
\inf_{x,y} P_y[I_f(\alpha)] \stackrel{\alpha \to \infty}{\longrightarrow} 1 , \text{ where } I_{f}(\alpha) \stackrel{\text{def.}}{=} \Big\{  \int_0^{\infty} V_f(X_s) \mathrm{d}s \geq -\alpha \Big\}.
\end{equation}
We start by explaining how \eqref{eq:beta-1} yields the claim. To this effect, first note that uniformly in $x \in \Z^d$, $z \in \Gamma(x)$ and $y \in \partial^{\text{out}} B_4^z$, abbreviating $B_i = B_i^z$, one has by \eqref{eq:B_i} (cf.~also \eqref{eq:LB-final-1} regarding $r_1$) that $P_{y}[\widetilde{H}_{B_3}=\infty] \geq \Cl[c]{c:escape-1}$. Applying \eqref{eq:beta-1} we then fix $\alpha$ large enough such that $P_y[I_f(\alpha)] \geq 1- \frac{\Cr{c:escape-1}}{2}$. 
Using that $f \leq C(u,v, \varepsilon)$, as can be seen by inspection of \eqref{eq:f1hatdef}, whence $$\frac{d\widetilde{P}_y^f}{dP_y} 1_{I_{f}(\alpha)} \geq \Cl[c]{c:escape-2}(u,v,\varepsilon),$$ it thus follows that (uniformly in $x,y$ as above)
\begin{multline}
\label{eq:beta-2}
\widetilde{P}_{y}^f[{H}_{B_3}=\infty] \geq 
\widetilde{P}_y^f[{H}_{B_3} = \infty, \, I_f(\alpha)]  \\
\geq \Cr{c:escape-2} ({P}_y[{H}_{B_3} = \infty] -P_y [I_f(\alpha)^c] )  \geq  {\Cr{c:escape-1}\Cr{c:escape-2}}/{2} \, {\stackrel{
\text{def.}}=\Cl[c]{c:escape-yes}.}
\end{multline} 
Now to \eqref{eq:beta-0}. Let ${\beta}^{f}_{0}(\cdot)$ be defined as ${\beta}^{f}(\cdot)$ in \eqref{eq:beta-0} but with $\infty$ replaced by $T_{B_4^z}$, the exit time from $B_4^z$. We first argue that
\begin{align}\label{eq:beta-3}
\beta^{f}_0(x) \to 0 \text{ as } |x| \rightarrow \infty.
\end{align}
To see this, first observe that, prior to exiting $B_4=B_4^z$, the law of the tilted random 
walk is identical to that of the simple random walk (indeed $B_4 \subset T^1 \cup 
(T^5\setminus T^3)$ whenever $z\in \Gamma$, and $f$ is constant on this set; 
cf.~\eqref{eq:chosenf-2}). In particular this implies for $y$ as in \eqref{eq:beta-0} that 
$\widetilde{P}_y^f[H_{B_2} < T_{B_4}]= {P}_y[H_{B_2} < T_{B_4}] $ and the latter is 
bounded by ${P}_y[H_{B_2} < \infty] \leq r_1^{-c}$ using standard arguments and the 
choice of boxes in \eqref{eq:B_i}. All in all, \eqref{eq:beta-3} thus follows. 

Consider now the quantity $\widetilde{P}_y^f[T_{B_4} \leq H_{B_2} < \infty]$, for $y \in \partial B_3$. Applying the strong Markov property at times $T_{B_4}$ and $H_{B_3} \circ \theta_{T_{B_4}}$, where $\theta_t$ denotes the canonical shift by $t>0$, it readily follows that
$$
\widetilde{P}_y^f[T_{B_4} \leq H_{B_2} < \infty] \leq \sup_v \widetilde{P}_{v}^f[{H}_{B_3}<\infty] \cdot \bigg( \beta_0^f(x) + \sup_{w} \widetilde{P}_{w}^f[T_{B_4} \leq H_{B_2} < \infty] \bigg),
$$
with the supremum over $v$ ranging over $\partial^{\text{out}} B_4$ and $w$ over $\partial B_3$. Taking a supremum over $y \in \partial B_3 $ on the left-hand side, using 
\eqref{eq:beta-2} to bound $\sup_v \widetilde{P}_{v}^f[{H}_{B_3}<\infty] \leq 1- 
\Cr{c:escape-yes}$ and collecting terms proportional to 
$\sup_{w} \widetilde{P}_{w}^f[T_{B_4} \leq H_{B_2} < \infty]$ yields that
\begin{equation} \label{eq:beta-4}
\sup_{w} \widetilde{P}_{w}^f[T_{B_4} \leq H_{B_2} < \infty] \leq \Cr{c:escape-yes}^{-1}(1- \Cr{c:escape-yes})\beta_0^f(x).
\end{equation}
The claim \eqref{eq:beta-0} now follows from \eqref{eq:beta-3} and \eqref{eq:beta-4}.

\medskip

It thus remains to prove \eqref{eq:beta-1}. For real-valued $V$ let $V^-=(-V) \vee 0$. In the rest of the proof constants $c,C,\dots$ may freely depend on all of $u,v$ and $\varepsilon$ (as entering the definition of $f$ in \eqref{eq:f1hatdef}). We will show that for all  $x \in \Z^d$ and $y \in \bigcup_{z\in \Gamma} \partial^{\text{out}} B_4^z$,
\begin{equation}\label{eq:beta-5}
E_y\Big[\int_0^{\infty} V_f^-(X_s) ds\Big] \leq C.
\end{equation}

We first seek to simplify the expectation in \eqref{eq:beta-5} insofar as possible. Recalling the form of $f$ from \eqref{eq:f1hatdef}, see also \eqref{eq:chosenf-2} and 
exploiting the fact that $f$ is harmonic at $z \in \mathbb{Z}^d$ unless $z \in K$, where $K= \partial T^2 \cup \partial^{\text{out}} T^3 \cup \partial T^6 \cup \partial^{\text{out}} U,$ it follows that 
$V_{f}(z)$ vanishes (thus also $V_{f}^-(z)$) unless $z$ belongs to this set, whence
\begin{equation}
	E_y\Big[\int_0^{\infty} V_{f}^-(X_s) ds\Big] = E_y\Big[\int_0^{\infty} V_{f}^-(X_s) 1\{X_s \in K\}ds\Big] 
	= \sum_{z\in K} V_{f}^-(z) g(y,z). \label{eq:intVboundproof} \end{equation}
Abbreviating $\Cl[c]{c:f_0}= 1 - ({\frac{u_*-\varepsilon}{u}})^{\frac12} \in (0,1)$, 
$\Cl[c]{c:f_1-1}= ({\frac{u_{*}+\varepsilon}{v}})^{1/2} -1$ and $\Cl[c]{c:f_1-2}= \Cr{c:f_0}+  \Cr{c:f_1-1} $, so that $f = 1 + \Cr{c:f_1-1} h_2 - \Cr{c:f_1-2} h_1$, we have that
\begin{align} \label{eq:beta-6}
V_{f}(z) = - \frac{\Delta f(z)}{f(z)} = \frac{\Cr{c:f_1-2}\Delta h_1(z) - \Cr{c:f_1-1} \Delta h_2(z)}{1 + \Cr{c:f_1-1} h_2(z) - \Cr{c:f_1-2} h_1(z)}, \ z \in \Z^d.
\end{align}
 Recalling that $h_1(z) = P_z[H_{T^2} < T_{T^3}]$, $h_2(z) = P_z[H_{T^6} < T_{U}]$, it follows that for $z \in \partial^{\text{out}} T^3$, we have $h_2(z) = 1, \Delta h_2(z) = 0, 
 h_1(z) = 0$, so that $V_{f}(z) \geq 0$ by inspection of \eqref{eq:beta-6}, whence 
 $V_{f}^-(z)=0$. The same fate applies to $z \in \partial T^6$, using now that $h_1(z)=0, 
 \Delta h_1(z) = 0, h_2(z) = 1$. Thus, \eqref{eq:intVboundproof} boils down to
 \begin{equation}
	E_y\Big[\int_0^{\infty} V_{f}^-(X_s) ds\Big]	= \sum_{z\in \partial T^2 \cup  \partial^{\text{out}} U} V_{f}^-(z) g(y,z). \label{eq:intVboundproof'} \end{equation}
 Meanwhile, $h_1(z) = h_2(z) = 1$ and $\Delta h_2(z) = 0$ whenever $z \in \partial T^2$ and $|x| > C(\delta)$ (see below \eqref{eq:f1hatdef}), hence \eqref{eq:beta-6} yields for 
 such $z$ that
 \begin{equation}
 V_{f}^-(z) = -V_{f}(z) = \frac{\Cr{c:f_1-2}}{1-\Cr{c:f_0}} \frac1{2d}\sum_{z'\sim z} P_{z'}[H_{T^2} \geq  T_{T^3}] = \frac{\Cr{c:f_1-2}}{1-\Cr{c:f_0}} P_{z}[\widetilde{H}_{T^2} \geq  T_{T^3}], 
 \label{eq:stepsBoundingVfromBelow1}\end{equation}
using the simple Markov property in the last step. By choice of $T^3$ above \eqref{eq:T_i-U}, one readily infers via Proposition~\ref{L:hit} that $\inf_{\xi \notin T^3 }P_{\xi}[H_{T^2} = \infty] \geq c$, and applying the strong Markov property (or appealing directly to \eqref{eq:cap-10}), this is straightforwardly seen to imply that $P_{z}[\widetilde{H}_{T^2} \geq  T_{T^3}] \leq c^{-1} e_{T^2}(z)$ for all $z \in \partial T^2$, which in turn yields that
\begin{equation}
\label{eq:expec-bdI}
 \sum_{z\in \partial T^2 } V_{f}^-(z) g(y,z) \leq C \sum_{z\in \partial T^2 }  g(y,z) e_{T^2}(z) \stackrel{\eqref{eq:last-exit}}{\leq} C.
\end{equation}

Going back to \eqref{eq:beta-6}, for $z \in \partial^{\text{out}} U$, we have $h_1(z)=0, \Delta h_1(z) = 0, h_2(z) = 0$, so that
\begin{align} \label{eq:beta-7}
V_{f}^-(z) = \frac{\Cr{c:f_1-2}}{2d}\sum_{z'\sim z} P_{z'}[H_{T^6} < T_{U}] = \Cr{c:f_1-2} P_z[H_{T^6} < \widetilde{H}_{U^c}].
\end{align}
By construction, see \eqref{eq:T_i} and \eqref{eq:T_i-U}, $T^6 \subset U/2$, where with 
hopefully obvious notation $U/2$ is the box concentric to $U$ with half the radius. Thus 
bounding $P_z[H_{T^6} < \widetilde{H}_{U^c}] \leq P_z[H_{U/2} < \widetilde{H}_{U^c}]$ and 
by considering the projection of the walk in the direction orthogonal to the face to which $z \in 
\partial^{\text{out}} U$ belongs, a straightforward Gambler's ruin argument readily yields that 
$P_z[H_{T^6} < \widetilde{H}_{U^c}] \leq C|x|^{-1}$, 
hence,
 \begin{align}
\sum_{z \in \partial^{\text{out}} U} V_{f}^-(z) g(y,z) \leq \frac{C}{|x|} \sum_{z \in \partial^{\text{out}} U} g(y,z)  \leq \frac{C'}{|x|} |{\partial^{\text{out}} U}| |x|^{2-d} \leq C'',
\label{eq:VUNapprox}\end{align}
where we have used that $|y-z| \geq c|x|$ and a standard estimate on the Green's function. 

Overall, feeding \eqref{eq:expec-bdI} and \eqref{eq:VUNapprox} into \eqref{eq:intVboundproof'} completes the verification of \eqref{eq:beta-5} for all $f$ as in \eqref{eq:f1hatdef}, and with it the verification of 
\eqref{eq:beta-1}, thus concluding the proof.
\end{proof}

As an immediate consequence of Lemma~\ref{lem:tilted_beta_bound}, we obtain the following comparison between tilted and untilted entrance laws. Extending the notation $h_{B}(\cdot,\cdot)$, $\tilde h_{B}(\cdot,\cdot)$, from \eqref{eq:h_B}, we define for $B \subset B'\subset \Z^d$ the quantity $\tilde{h}_{B,B'}(x,y) = \widetilde{P}_x^f[H_B < T_{B'}, X_{H_B} = y]$ for $x,y \in \Z^d$, and similarly ${h}_{B,B'}(x,y)$ with $P_x$ in place of $\widetilde{P}_x^f$. Thus, $\tilde{h}_{B}= \tilde{h}_{B,\Z^d}$. We will sometimes write $\tilde{h}_{B,B'}^f= \tilde{h}_{B,B'}$ or $\tilde{h}_{B}^f= \tilde{h}_{B}$ to insist on the dependence on $f$.

\begin{corollary} \label{lem:equil_potential_comparisons_collated} For all $f$ as in \eqref{eq:f1hatdef}, $\epsilon_0 > 0$, $\delta \in (0,\frac16)$, $|x| \geq C(\epsilon_0,\delta)$, $z \in \Gamma(x)$, $y \in \partial B_2^z$ and $y'\in  \Z^d$, abbreviating $B_i=B_i^z$, 
\begin{align}
h_{B_1,B_3}(y,y') \leq \tilde{h}_{B_1}^f(y,y') \leq (1+\epsilon_0)\min_{\tilde y \in \partial B_2} h_{B_1,B_3}(\tilde y,y').
\label{eq:equil_comparison_new}
\end{align}
\end{corollary}
\begin{proof}
The first bound in \eqref{eq:equil_comparison_new} is immediate since $\tilde{h}_{B_1}^f(y,y') \geq \tilde{h}_{B_1,B_3}^f(y,y') = {h}_{B_1,B_3}(y,y') $ for $y \in \partial B_2^z$, using that $X_{\cdot \wedge T_{B_4}}$ has the same law under $P_y$ and $\widetilde{P}_y^f$, cf.~below \eqref{eq:beta-3}. For the second bound in \eqref{eq:equil_comparison_new}, applying a similar reasoning as used to deduce \eqref{eq:beta-4} yields that 
\begin{align}
 \tilde{h}_{B_1}^f(y,y') \leq \frac{1}{1-\beta^f(x)} \max_{\tilde y \in \partial B_2} \tilde{h}_{B_1,B_3}^f(\tilde y,y'),
\end{align}
and the latter quantity equals $\max_{\tilde y} {h}_{B_1,B_3}(\tilde y,y')$. To conclude \eqref{eq:equil_comparison_new}, one applies Lemma~\ref{lem:tilted_beta_bound} and uses the fact that $\max_{\tilde y} {h}_{B_1,B_3}(\tilde y,y') \leq (1+ |x|^{-c\delta})\min_{\tilde y} {h}_{B_1,B_3}(\tilde y,y')$, which follows by similar considerations as e.g.~in \cite[Lemma 3.5]{zbMATH06257634}.
\end{proof}

Corollary~\ref{lem:equil_potential_comparisons_collated} will readily yield the part of Lemma~\ref{L:tilt-compa} concerning $\frac{\tilde{\pi}}{\pi}$. Dealing with $\frac{\tilde{\pi}_0}{\pi_0}$, cf.~\eqref{eq:exc2} requires a control on the tilted equilibrium measure, which is the object of the next lemma. This result will also be needed in the course of proving Lemma~\ref{L:N-conc}. In the sequel it will be convenient to introduce ``effective'' levels, for $f$ as in \eqref{eq:f1hatdef},
\begin{equation} \label{eq:utildefefinition}
\begin{split}
&\tilde{u}_f^z = 
\begin{cases}
u & \text{if } z\in \Gamma_{\text{int}} \\
v & \text{if } z\in \Gamma_{\text{ext}}
\end{cases} \\
&u_f^z = f^2(z)\tilde{u}_f^z \stackrel{\eqref{eq:chosenf-2}}{=} 
\begin{cases}
u_*-\varepsilon & \text{if } z\in \Gamma_{\text{int}} \\
u_*+ \varepsilon & \text{if }  z\in \Gamma_{\text{ext}}.
\end{cases} 
\end{split}
\end{equation} 
Note that \eqref{eq:utildefefinition} defines $\tilde{u}_f^z$ and ${u}_f^z $ for any $z \in \Gamma$ since $\Gamma = \Gamma_{\text{int}}\cup \Gamma_{\text{ext}}$ according to our definition in the previous section.\\

\begin{proposition}\label{P:harm-comp} For all $\epsilon_0 \in (0,1)$, $f$ as in \eqref{eq:f1hatdef}, $x \in \mathbb{Z}^d$, $z \in \Gamma(x)$ and $|x| \geq C(\epsilon_0,\delta)$,
\begin{align}
&(1-\epsilon_0) u_f^z\cdot \textnormal{cap}(B_2^z) \leq \tilde u_f^z \cdot \widetilde{\textnormal{cap}}_f(B_2^z) \leq (1+\epsilon_0) u_f^z \cdot\textnormal{cap}(B_2^z), \text{ and } \label{eq:cap_comparison}\\
&(1-\epsilon_0) u_f^z \cdot e_{B_1^z}(y) \leq \tilde u_f^z \cdot  \tilde e_{B_1^z}(y) \leq (1+\epsilon_0) u_f^z \cdot  e_{B_1^z}(y), \ y \in \Z^d. \label{eq:equi-comparison}
\end{align}
\end{proposition}

\begin{proof} A bound similar to the first inequality in \eqref{eq:cap_comparison} was proved in \cite[Proposition 3.1]{zbMATH06257634}, and can be obtained within the present setup via similar arguments, essentially with Lemma~\ref{lem:tilted_beta_bound} now playing the role of \cite[Lemma 3.3]{zbMATH06257634}.

We now focus on the second inequality in \eqref{eq:cap_comparison}, which requires additional arguments. We begin by noting that, with the effective levels defined in \eqref{eq:utildefefinition} and the tilted Green's function $\tilde g$ as introduced in \eqref{eq:g-tilted}, one has the identity
\begin{equation}
\label{eq:cap_comparison-pf1}
\tilde u_f^z \sum_{y, y' \in B_2}\tilde e_{B_2}(y) \tilde g(y,y') f^2(y') =  u_f^z \sum_{y, y' \in B_2} e_{B_2}(y)  g(y,y'),
\end{equation}
where $g=g_{\Z^d}$ is the usual Green's function given by \eqref{eq:g_U}.  To see \eqref{eq:cap_comparison-pf1}, one simply applies \eqref{eq:greenseqidentity} to the left-hand side and notices that $f^2(y')=f^2(z)$ for all $y' \in B_2^z$ and similarly \eqref{eq:last-exit} to the right-hand side, to conclude that both sides equal $u_f^z |B_2|$. Roughly speaking, we aim to argue that the left-hand side is an upper bound for $\tilde u_f^z \cdot \widetilde{\textnormal{cap}}_f(B_2^z) $, while the right-hand side is a lower bound for $u_f^z \cdot\textnormal{cap}(B_2^z)$. The desired inequality will then follow by means of \eqref{eq:cap_comparison-pf1}. 

To this end, using the fact that $\tilde g_{B_3}(y,y')= g_{B_3}(y,y')$ for any $y,y' \in B_3$ we bound
\begin{equation}\label{eq:cap_comparison-pf2}
\begin{split}
&\sum_{y, y' \in B_2}\tilde e_{B_2}(y) \tilde g(y,y') \geq \widetilde{\textnormal{cap}}_f(B_2) 
\inf_{y  \in B_2}\sum_{y'\in B_2} g_{B_3}(y,y') \\
= \, & \widetilde{\textnormal{cap}}_f(B_2) \inf_{y  \in B_2}\sum_{y'\in B_2} \big( g(y,y') - 
E_y\big[g(X_{T_{B_3}},y')\big]\big) \\
\geq \, & \widetilde{\textnormal{cap}}_f(B_2) \inf_{y  \in 
B_2}\sum_{y'\in B_2} \big( g(y,y') - C|x|^{-c\delta} \sup_{ y''  \in B_2} g(y'',y')\big), 
\end{split}
\end{equation} 
for all $|x| \ge C(\delta)$. Using the fact that uniformly in $y \in B_2$,
\begin{equation}\label{eq:cap_comparison-pf3}
\sum_{y' \in B_2} g(y,y') = \Cl[c]{c:g-L1} |B_2|^{\frac2d} (1+ o(1)), \text{ as $|x|\to \infty$},
\end{equation}
for a suitable constant $\Cr{c:g-L1}$, see e.g.~\cite[Lemma 1.1.]{zbMATH06257634} for a proof, one immediately deduces by inserting \eqref{eq:cap_comparison-pf3} into \eqref{eq:cap_comparison-pf2} and using that $f(y')$ is constant in $B_2$ that the left-hand side of \eqref{eq:cap_comparison-pf1} is bounded from below by $\Cr{c:g-L1} \tilde u_f^z f^2(z)\cdot \widetilde{\textnormal{cap}}_f(B_2) |B_2|^{\frac2d} (1-\epsilon_0)$, for all $|x| \geq C(\delta,\epsilon_0)$. Employing \eqref{eq:cap_comparison-pf3} allows to bound the right-hand side of \eqref{eq:cap_comparison-pf1} from above by $\Cr{c:g-L1}  u_f^z \cdot\widetilde{\textnormal{cap}}_f(B_2) |B_2|^{\frac2d} (1+\epsilon_0)$. The claim now follows using \eqref{eq:cap_comparison-pf2} and the fact that $u_f^z = f^2(z)\tilde{u}_f^z$, see \eqref{eq:utildefefinition}. This completes the proof of \eqref{eq:cap_comparison}.

The inequalities \eqref{eq:equi-comparison} are a consequence of \eqref{eq:cap_comparison} and \eqref{eq:equil_comparison_new}, as we now explain. Indeed, to obtain the desired upper bound for $u_f^z \cdot  \tilde e_{B_1}(y)$, applying the sweeping identity (see, e.g.~(1.12) in \cite{MR3417515}) to the sets $B_1 \subset B_2$ yields that for all $|x| \geq C'(\epsilon_0, \delta)$,
\begin{align*}
\tilde e_{B_1}(y) &= \sum_{y' \in \partial B_2} \tilde e_{B_2}(y') \tilde h_{B_1}(y'y) \stackrel{\eqref{eq:equil_comparison_new}}{\leq} (1+ { \epsilon_0}/{2})\widetilde{\textnormal{cap}}_f(B_2) \min_{\tilde y' \in \partial B_2} h_{B_1,B_3}(\tilde y',y) \\ &\stackrel{\eqref{eq:cap_comparison}}{\leq}  (1+ { \epsilon_0}) \frac{ u_f^z}{\tilde u_f^z}\textnormal{cap}(B_2^z)  \min_{\tilde y' \in \partial B_2} h_{B_1}(\tilde y',y)\leq  (1+ { \epsilon_0}) \frac{ u_f^z}{\tilde u_f^z}
\sum_{\tilde y' \in \partial B_2} e_{B_2}(y')h_{B_1}(\tilde y',y) \\
&=  (1+ { \epsilon_0}) \frac{ u_f^z}{\tilde u_f^z} e_{B_1}(y),
\end{align*}
where the last step uses again the sweeping identity. The lower bound on $\tilde e_{B_1}$ in \eqref{eq:equi-comparison} is obtained similarly.
\end{proof}

\begin{proof}[Proof of Lemma~\ref{L:tilt-compa}]
We first deal with the part of \eqref{eq:tilt-compa1} concerning $\frac{\tilde{\pi}_0}{\pi_0}$. Summing \eqref{eq:equi-comparison} over $B_1^z$ and subsequently using this resulting inequality when dividing by the capacity of $B_1^z$ yields an analogue of \eqref{eq:tilt-compa1} regarding normalized (tilted and untilted) equilibrium measures. In view of the definition of $\pi_0$ in \eqref{eq:exc2}, the claim follows.

We now show the part of \eqref{eq:tilt-compa1} concerning $\frac{\tilde\pi}{\pi}$. To this effect, we first observe that, for all $\epsilon_0 > 0$, $\delta \in (0,\frac16)$ and $|x| \geq C(\epsilon_0,\delta)$,
 uniformly in $z \in \Gamma(x)$, $y \in \partial B_2^z$ and $y'\in  \partial B_1^z$,
\begin{align}
(1-\epsilon_0) \bar{e}_{B_1}(y') \leq \frac{\tilde{h}_{B_1}^f(y,y')}{\widetilde{P}_y^f[H_{B_1} <\infty]} \leq (1+\epsilon_0) \bar{e}_{B_1}(y')  \label{eq:tildehB1equilmes}
\end{align}
where $f$ is as in \eqref{eq:f1hatdef}; indeed \eqref{eq:tildehB1equilmes} follows readily from Corollary~\ref{lem:equil_potential_comparisons_collated}, as we now explain. By applying \eqref{eq:equil_comparison_new}, both in its given form and when summing over $y'$, one deduces that the ratio in \eqref{eq:tildehB1equilmes} and its untilted analogue $\frac{{h}_{B_1}(y,y')}{{P}_y[H_{B_1} <\infty]} $ are comparable up to multiplicative errors of order $1+ O(\epsilon_0)$ when $|x| \geq C(\epsilon_0,\delta)$. Then one uses the fact that the untilted analogue of \eqref{eq:tildehB1equilmes}, i.e.~bounding $\frac{{h}_{B_1}(y,y')}{{P}_y[H_{B_1} <\infty]} $ from above and below by $(1\pm\epsilon_0) \bar{e}_{B_1}(y')$ is classically known, see, for example,  \cite[Theorem 2.1.3]{Law91}. Hence, overall, \eqref{eq:tildehB1equilmes} follows.

Now let $\zeta,\zeta' \in \Xi$ be two excursions between $B=B_1^z$ and $U^c$, where $U= B_2^z$, cf.~around \eqref{eq:muB1intro} for notation. Rather than dealing with the ratio $\frac{\tilde\pi(\zeta,\zeta')}{\pi(\zeta,\zeta')}$ directly, we will separately consider $\frac{\tilde\pi(\zeta,\zeta')}{\bar{e}_{B_1} (\zeta_0')}$ and $\frac{\pi(\zeta,\zeta')}{\bar{e}_{B_1} (\zeta_0')}$, with $\zeta_0' \in \partial B_1^z$ denoting the starting point of $\zeta'$. Recalling $\tilde{\pi}$ from \eqref{eq:exc2}, it follows, abbreviating $y = \zeta_n$, $y'=\zeta_0'$ and with the aid of \eqref{eq:tildehB1equilmes} that
\begin{align*}
\frac{\tilde\pi(\zeta,\zeta')}{\bar{e}_{B_1} (y')} &= \frac{\tilde h_{B_1}^f(y, y')}{\bar{e}_{B_1}(y')} + 1-h_{B_1}(y) \leq 1 + \epsilon_0 h_{B_1}(y),
\end{align*}
along with a similar lower bound, implying overall that
\begin{align}\label{eq:pi_final}
\bigg| \frac{\tilde\pi(\zeta,\zeta')}{\bar{e}_{B_1} (y')} -1 \bigg| \leq \epsilon_0 h_{B_1}(y).
\end{align}
The same bound as \eqref{eq:pi_final} is obtained for $\pi$ instead of $\tilde{\pi}$ using the (classical) untilted analogue of \eqref{eq:tildehB1equilmes}, see \cite[Theorem 2.1.3]{Law91}. From \eqref{eq:pi_final} and its version for $\pi$, the desired bound on $|\frac{\tilde{\pi}}{\pi}-1|$ readily follows with $\epsilon_0=1$.
\end{proof}

The rest of this section is geared towards the proof of Lemma~\ref{L:N-conc}, which concerns the random variables $\mathcal{N}^u$ introduced in \eqref{eq:N} and their tilted analogue. We first isolate the following result. Under $P_x$, define the successive return times to $B=B_1^z$ and departure times from $U=B_2^z$ as $R_1= H_B$, and for $k \geq 1$, $D_k= R_k+ T_U\circ R_k$, $R_{k+1}= D_k + H_B \circ D_k$  (assuming $R_k$ is finite, and else set $D_k=R_{k+1}=\infty$). Now let 
\begin{equation}
\label{eq:tau-1}
\tau= \tau_B= \sup\{ k \geq 1: D_k< \infty\}.
\end{equation}
The random variable $\tau$ counts the number of excursions between $B$ and $U^c$ made by the walk.
We write $\tilde\tau$ for its pendant defined under $\widetilde{P}_x^f$. Recall $\tilde\pi_0$ from \eqref{eq:exc2}.

\begin{lemma} \label{L:tau} For all $ \lambda \in (-\infty, c)$, $f$ as in \eqref{eq:f1hatdef}, $x \in \Z^d$, $z \in \Gamma(x)$, with $\tau=\tau_{B_1^z}$,
\begin{equation}\label{eq:tau-2}
e^{\lambda}(1- \beta^{f}(x)) \leq  \widetilde{E}_{\tilde\pi_0}^f\big[e^{\lambda \tau} \big] \leq \frac{e^{\lambda}(1- \beta^{f}(x))}{1-e^{\lambda} \beta^{f}(x)}, \quad \text{(see \eqref{eq:beta-0} for $\beta^{f}(x)$)}
\end{equation}
\end{lemma}
\begin{proof}
The upper bound is proved in a similar way as \cite[Lemma 2.7]{MR3602841}. The lower bound is obtained by bounding $\widetilde{E}_{\tilde\pi_0}^f\big[e^{\lambda \tau} \big]  \geq e^{\lambda} \widetilde{P}_{\tilde\pi_0}^f[\tau=1 ] $.
\end{proof}
It remains to give the:

\begin{proof}[Proof of Lemma~\ref{L:N-conc}] In view of \eqref{eq:tau-1}, and by inspection of \eqref{eq:sigma-k}, one observes that if one defines recursively $\hat\sigma_1=1$ and for $ k \geq 1$ $\hat\sigma_{k+1}=\inf \{k \geq  \hat\sigma_k : \sigma_k =1\}$, then $\hat{\sigma}_{k+1}-\hat\sigma_k$ has the same law as $\tau$ under $P_{e_B}$. Moreover, the random variables $\hat{\sigma}_{k+1}-\hat\sigma_k$, $k \geq 1$ are independent and analogous statements hold for tilted quantities. It follows in view of \eqref{eq:N} that for all $v>0$
\begin{equation}
\label{eq:conc-10}
\widetilde{\mathcal{N}}^v \stackrel{\text{law}}{=}\sum_{i=1}^{\widetilde{\Theta}(u)} \tilde{\tau}_i
\end{equation}
where $ \tilde{\tau}_i$, $i \geq 1$, are i.i.d.~with same law as $\tilde{\tau}$ under $\widetilde{P}^f_{\tilde\pi_0}$, where $\tilde\pi_0$ is the normalized tilted equilibrium measure on $B$, cf.~\eqref{eq:exc2}, and $\widetilde{\Theta}(u)$ is an independent Poisson variable with mean $u \cdot \widetilde{\text{cap}}(B)$. A representation similar to \eqref{eq:conc-10} can be derived for ${\mathcal{N}}^u$.

We now focus on \eqref{eq:N-conc1}; the remaining bounds are obtained similarly. Thus let  
$f=f^{v,u; \varepsilon}$, $v=u(1-\eta \sqrt{\delta})$ and $B=B_1^z$ for 
some $z \in \Gamma_{\textnormal{ext}}$. From \eqref{eq:conc-10}, Lemma~\ref{L:tau} and 
Proposition~\ref{P:harm-comp}, we infer that for $\lambda \in (-\infty,c)$, $\epsilon_0 \in (0, 
1)$ and $|x| \geq C(\epsilon_0)$,
\begin{multline}
\label{eq:conc-11}
\log \widetilde{\E}^f \big[e^{\lambda \cdot \widetilde{\mathcal{N}}^v} \big] = v \cdot \widetilde{\text{cap}}(B) \big (  \widetilde{E}_{\tilde\pi_0}^f\big[e^{\lambda \tau} \big] -1 \big) \stackrel{\eqref{eq:tau-2}}{\leq}
v \cdot \widetilde{\text{cap}}(B) \frac{e^{\lambda} -1}{1-e^{\lambda} \beta^{f}(x)}
\\ \stackrel{\eqref{eq:equi-comparison}}{\leq} (u_*+ \varepsilon) \text{cap}(B) \frac{(1+\epsilon_0)(e^{\lambda} -1)}{1-e^{\lambda} \beta^{f}(x)},
\end{multline}
where, in applying \eqref{eq:equi-comparison}, we have summed over $y \in \Z^d$ and used that $v  \cdot \frac{u_f^z}{\tilde{u}_f^z}=  u_*+ \varepsilon $ by \eqref{eq:utildefefinition} (and using that $z \in \Gamma_{\textnormal{ext}}$ with $f$ as above). A lower bound corresponding to \eqref{eq:conc-11} can be derived similarly, along with similar estimates for $\log {\E} [e^{\lambda \cdot {\mathcal{N}}^w}]$, $w > 0$, obtained by  means of an obvious analogue of Lemma~\ref{L:tau} (but no longer requiring Proposition~\ref{P:harm-comp}).  Using that $(e^{\lambda}- 1) \vee (1-e^{-\lambda})  \leq \lambda (1+ C\lambda)$ for $0 \leq \lambda \leq 1$ and applying Chebyshev's inequality separately to $\mathcal{N}^{u_*+\frac{1}{2}\varepsilon}$, $\widetilde{\mathcal{N}}^v$ and $\mathcal{N}^{u_*+\frac{3}{2}\varepsilon}$, selecting in each case $\lambda = \epsilon_0 = c \varepsilon $ in \eqref{eq:conc-11}, while using Lemma~\ref{lem:tilted_beta_bound} to control $\beta^f$ in \eqref{eq:conc-11}, one ensures that with probability at least $1 -e^{-\tilde cr_0^{\tilde c}}$ and for $|x|\geq C(\varepsilon,\delta)$, the inequalities
\begin{align*}
&{\mathcal{N}}^{u_*+\frac{1}{2}\varepsilon} \leq \text{cap}(B) (u_*+ \textstyle \frac{5}{8}\varepsilon), \\
& \text{cap}(B) (u_*+ \textstyle \frac{7}{8}\varepsilon)  \leq  \widetilde{\mathcal{N}}^v \leq  \text{cap}(B) (u_*+ \textstyle \frac{9}{8}\varepsilon), \\
&{\mathcal{N}}^{u_*+\frac{3}{2}\varepsilon} \geq \text{cap}(B) (u_*+ \textstyle \frac{11}{8}\varepsilon)
\end{align*}
all hold. From this  \eqref{eq:N-conc1} immediately follows for sufficiently small choice of $c \in (0,1)$.
\end{proof}

\textbf{Acknowledgments.} We thank Yinfei Zeng and two anonymous referees for their 
careful reading and pertinent comments on a previous version of this article.
This work was supported by EPSRC grant EP/Z000580/1. SG's research is partially 
supported by the SERB grant SRG/2021/000032 and a grant from the Department of 
Atomic Energy, Government of India, under project 12-R\&D-TFR-5.01-0500. The research of PFR is supported by the European Research Council (ERC) under the European Union's Horizon Europe research and innovation program (grant agreement No 101171046). The research of YS was supported by EPSRC Centre for Doctoral Training in Mathematics of Random 
Systems: Analysis, Modelling and Simulation (EP/S023925/1).

{
	\bibliography{biblicomplete}
	\bibliographystyle{abbrv}}
\end{document}